\documentclass[11pt, letterpaper]{article}

\usepackage{amsfonts}
\usepackage{amssymb}
\usepackage{amsmath}
\usepackage{amsthm}
\usepackage{latexsym}
\usepackage{color}

\def\r{\mathbb R}
\def\h{\mathbb H}
\def\s{\mathbb S}
\def\z{\mathbb Z}
\def\n{\mathbb N}

\def\c{\mathbb C}

\newcommand{\dvol}{\,dvol}
\DeclareMathOperator{\Hyperg}{Hyperg}
\DeclareMathOperator{\residue}{Res}

\DeclareMathOperator{\divergence}{div}

\newtheorem{theorem}{Theorem}[section]
\newtheorem{definition}[theorem]{Definition}
\newtheorem{proposition}[theorem]{Proposition}

\newtheorem{lemma}[theorem]{Lemma}
\newtheorem{corollary}[theorem]{Corolary}
\numberwithin{equation}{section}
\theoremstyle{remark}
\newtheorem{remark}[theorem]{Remark}

\begin{document}

\title{Isolated singularities for a semilinear equation for the \\fractional Laplacian arising in conformal geometry}

\author{DelaTorre, Azahara\footnote{Universitat Polit\`ecnica de Catalunya, ETSEIB-MA1, Av. Diagonal 647, 08028 Barcelona, Spain, azahara.de.la.torre@upc.edu. Supported by MINECO grants MTM2011-27739-C04-01 and  FPI-2012 fellowship. The author is part of the Catalan research group 2014SGR1083.}\\
Gonz\'alez, Mar\'ia del Mar\footnote{Universitat Polit\`ecnica de Catalunya, ETSEIB-MA1, Av. Diagonal 647, 08028 Barcelona, Spain, mar.gonzalez@upc.edu. Supported by MINECO grant MTM2011-27739-C04-01. The author is part of the Catalan research group 2014SGR1083.}}

\date{}
\maketitle

\abstract{We introduce the study of isolated singularities for a semilinear equation involving the fractional Laplacian. In conformal geometry, it is equivalent to the study of singular metrics with constant fractional curvature. Our main ideas are: first, to set the problem into a natural geometric framework, and second, to perform some kind of phase portrait study for this non-local ODE.}

\section{Introduction and statement of results}

We consider the problem of finding radial solutions for the fractional Yamabe problem in $\r^n$, $n\geq 2$, with an isolated singularity at the origin. This means to look for positive, radially symmetric solutions of
\begin{equation}\label{equation0}(-\Delta)^{\gamma}w=c_{n,\gamma}w^{\frac{n+2\gamma}{n-2\gamma}}\text{ in }\r^n \setminus\ \{0\},\end{equation}
where $c_{n,\gamma}$ is any positive constant that, without loss of generality, will be normalized as in Proposition \ref{cte}. Unless we state the contrary, $\gamma\in(0,\tfrac{n}{2})$. In geometric terms, given the Euclidean metric $|dx|^2$ on $\mathbb R^n$, we are looking for a conformal metric \begin{equation}\label{conformal-change}g_w=w^{\frac{4}{n-2\gamma}}|dx|^2,\ w>0,\end{equation} with positive constant fractional curvature $Q^{g_w}_\gamma\equiv c_{n,\gamma}$, that is radially symmetric and has a prescribed singularity at the origin.

Because of the well known extension theorem for the fractional Laplacian \cite{CaffarelliSilvestre,CaseChang,MarChang}  we can assert that equation \eqref{equation0} for the case $\gamma\in(0,1)$ is equivalent to the boundary reaction problem
\begin{equation}\label{equation1}\left\{
\begin{split}
-\divergence(y^{a}\nabla W)=0&\text{ in } \r^{n+1}_+,\\
W=w&\text{ on }\r^n\setminus\{0\},\\
-\tilde{d}_{\gamma}\lim_{y\rightarrow 0}y^a\partial_yW=c_{n,\gamma}w^{\frac{n+2\gamma}{n-2\gamma}}&\text{ on }\r^n\setminus\{0\}.
\end{split}\right.
\end{equation}
Here the constant $\tilde d_{\gamma}$ is defined in \eqref{dtg}. We note that it is possible to write $W=K_\gamma *_x w$, where $K_\gamma$ is the Poisson kernel \eqref{Poisson-kernel} for this extension problem.

It is known that $w_1(r)=r^{-\frac{n-2\gamma}{2}}$,
together with $W_1=K_\gamma *_x w_1$, is an explicit solution for \eqref{equation1}; this fact will be proved in Proposition \ref{cte} and as a consequence we will obtain the normalization of the constant $c_{n,\gamma}$.  Therefore, $w_1$ is the model solution for isolated singularities, and it corresponds to the cylindrical metric.

In the recent paper \cite{CaffarelliJinSireXiong} Caffarelli, Jin, Sire and Xiong characterize all the nonnegative solutions to \eqref{equation1}.
Indeed, let $W$ be any nonnegative solution of \eqref{equation1} in $\r^{n+1}_+$ and suppose that the origin 
 is not a removable singularity. Then, writing $r=|x|$ for the radial variable in $\mathbb R^n$, we must have that
  $$W(x,t)=W(r,t)\text{ and }\partial_r W(r,t)<0\quad \forall\ 0<r<\infty.$$
In addition, they also provide their asymptotic behavior. More precisely, if $w=W(\cdot,0)$ denotes the trace of $W$, then near the origin one must have that
\begin{equation}\label{asymptotics}
c_1r^{-\tfrac{n-2\gamma}{2}}\leq w(x)\leq c_2r^{-\tfrac{n-2\gamma}{2}},
\end{equation}
where $c_1$, $c_2$ are positive constants.

 We remark that if the singularity at the origin is removable, all the entire solutions for \eqref{equation1} have been completely classified by Jin, Li and Xiong  \cite{JinLiXiong} and Chen, Li and Ou \cite{ChenLiOu}, for instance. In particular, they must be the standard ``bubbles"
\begin{equation}\label{sphere}
w(x)=c\left(\frac{\lambda}{\lambda^2+|x-x_0|}\right)^{\frac{n-2\gamma}{2}},\quad c,\lambda>0, \ x_0\in\mathbb R^n.
\end{equation}

In this paper we initiate the study of positive radial solutions for \eqref{equation0}. It is clear from the above that we should look for solutions of the form
\begin{equation}\label{wv}
w(r)=r^{-\frac{n-2\gamma}{2}}v(r)\text{ on } \r^n\setminus\{0\},
\end{equation}
 for some  function $0<c_1\leq v\leq c_2$. In the classical case $\gamma=1$, equation \eqref{equation0} reduces to a standard second order ODE. However, in the fractional case \eqref{equation0} becomes a fractional order ODE, so classical methods cannot be directly applied here.

 The objective of this paper is two-fold: first, to use the natural interpretation of problem \eqref{equation0} in conformal geometry in order to obtain information about isolated singularities for the operator $(-\Delta)^\gamma$ from the scattering theory definition. And second, to take a dynamical system approach to explore how much of the standard ODE study can be generalized to the PDE \eqref{equation1}. In the forthcoming paper \cite{paper2}, that is a joint paper together also with M. del Pino and J. Wei, we conclude this study by using a variational method directly to construct solutions to \eqref{equation0} that generalize the well known Delaunay (sometimes called Fowler) solutions for the scalar curvature problem. Both papers complement each other.\\

Before stating our results we need to introduce the geometric setting. On a general Riemannian manifold $(M^n,g)$, the fractional curvature $Q^{g}_\gamma$ is defined from the conformal fractional Laplacian $P^{g}_\gamma$ as $Q^{g}_\gamma=P^{g}_\gamma(1)$, and it is a nonlocal version of the scalar curvature (which corresponds to the local case $\gamma=1$). The conformal fractional Laplacian is constructed from scattering theory on the conformal infinity $M^n$ of a conformally compact Einstein manifold $(X^{n+1},g^+)$ as a generalized Dirichlet-to-Neumann operator for the eigenvalue problem
\begin{equation}\label{scattering-introduction}-\Delta_{g^+}U-s(n-s)U=0\ \text{in}\ X, \quad s=\tfrac{n}{2}+\gamma,\end{equation}
and it is a (non-local) pseudo-differential operator of order $2\gamma$.
This construction is a natural one from the point of view of the AdS/CFT correspondence in Physics (\cite{AdS/CFT,Witten}).

 The main property of $P^g_\gamma$ is its conformal invariance; indeed, for a conformal change of metric $g_w=w^{\frac{4}{n-2\gamma}}g$, we have that
\begin{equation}\label{conformal-transformation}
P_\gamma^{g_w} (f)=w^{-\frac{n+2\gamma}{n-2\gamma}} P^{g}_\gamma(fw),\quad \text{for all } f\text{ smooth,}
\end{equation}
which, in particular when $f=1$, reduces to the fractional curvature equation
\begin{equation}\label{Qequation}
P_\gamma^{g}(w)=Q_\gamma^{g_w} w^{\frac{n+2\gamma}{n-2\gamma}}.
\end{equation}
The fractional Yamabe problem for equation \eqref{Qequation} on compact manifolds was considered in \cite{MarQing,MarWang}, while the fractional Nirenberg problem was introduced in \cite{JinLiXiong,JinLiXiongII,FangGonzalez}. In addition, the study of the singular version for the fractional Yamabe problem was initiated in \cite{GonzalezMazzeoSire}, where the authors construct model singular solutions. Here we look at the case of an isolated singularity.

We note that, in the Euclidean case, $P^{|dx|^2}_\gamma$ coincides with the standard fractional Laplacian $(-\Delta)^{\gamma}$, and thus, imposing the constant curvature condition in equation \eqref{Qequation} yields our original problem \eqref{equation0}. Moreover, one can check that the extension problem \eqref{equation1} comes from the scattering problem \eqref{scattering-introduction} when we take $g^+$ to be the hyperbolic metric $g^+=\frac{dy^2+|dx|^2}{y^2}$. See Section \ref{section:preliminaries} for a review of the known results on the fractional conformal Laplacian.\\

We present now the natural coordinates for studying isolated singularities of \eqref{equation0}. Let $M=\mathbb R^n\backslash \{0\}$ and use the Emden-Fowler change of variable $r=e^{-t}$, $t\in\mathbb R$; with some abuse of the notation, we write $v=v(t)$. Then, in radial coordinates, $M$ may be identified with the manifold $\mathbb R\times \mathbb S^{n-1}$, for which the Euclidean metric is written as
\begin{equation}\label{g0:introduction}|dx|^2=dr^2+r^2 g_{\mathbb S^{n-1}}=e^{-2t}[dt^2+g_{\mathbb S^{n-1}}]=:e^{-2t}g_0.\end{equation}
Since the metrics $|dx|^2$ and $g_0$ are conformally related, we prefer to use $g_0$ as a background metric and thus any conformal change \eqref{conformal-change} may be rewritten as
$$g_v=w^{\frac{4}{n-2\gamma}}|dx|^2=v^{\frac{4}{n-2\gamma}}g_0,$$
where we have used relation \eqref{wv}.
Looking at the conformal transformation property for $P^g_\gamma$ given in \eqref{conformal-transformation} and relation \eqref{wv} again, it is clear that
\begin{equation}\label{relation-introduction}
P_\gamma^{g_0}(v)=r^{\frac{n+2\gamma}{2}} P^{|dx|^2}_{\gamma}(r^{-\frac{n-2\gamma}{2}} v)=r^{\frac{n+2\gamma}{2}}
(-\Delta)^{\gamma} w,
\end{equation}
and thus the original problem \eqref{equation0} is equivalent to the following one: fixed $g_0$ as a background metric on $\mathbb R\times \mathbb S^{n-1}$, find a conformal metric $g_v=v^{\frac{4}{n-2\gamma}}g_0$ of positive constant fractional curvature $Q^{g_v}_\gamma$, i.e., find a positive smooth solution $v$ for
\begin{equation}\label{equation2}P_\gamma^{g_0}(v)=c_{n,\gamma} v^{\frac{n+2\gamma}{n-2\gamma}}\quad \text{on}\quad \mathbb R\times \mathbb S^{n-1}.\end{equation}

The point of view of this paper is to consider problem \eqref{equation2} instead of \eqref{equation0}, since it allows for a simpler analysis. In our first theorem we compute  the principal symbol of the operator $P_\gamma^{g_0}$ on $\mathbb R\times \mathbb S^{n-1}$ using the spherical harmonic decomposition for $\mathbb S^{n-1}$. With some abuse of notation, let $\mu_k=-k(k+n-2)$, $k=0,1,2,...$ be the eigenvalues of $\Delta_{\s^{n-1}}$, repeated according to multiplicity. Then, any function on $\mathbb R\times \mathbb S^{n-1}$ may be decomposed as $\sum_{k} v_k(t) E_k$, where $\{E_k\}$ is a basis of eigenfunctions. We show that the operator $P_\gamma^{g_0}$ diagonalizes under such eigenspace decomposition, and moreover, it is possible to calculate the Fourier symbol of each projection. Let
\begin{equation}\label{fourier}
\hat{v}(\xi)=\frac{1}{\sqrt{2\pi}}\int_{\r}e^{-i\xi \cdot t} v(t)\,dt
\end{equation}
be our normalization for the one-dimensional Fourier transform.
\begin{theorem}\label{th1}
 Fix $\gamma\in (0,\tfrac{n}{2})$ and let $P^k_{\gamma}$ be the projection of the operator $P^{g_0}_\gamma$ over each eigenspace $\langle E_k\rangle$. Then
$$\widehat{P_\gamma^k (v_k)}=\Theta^k_\gamma(\xi) \,\widehat{v_k},$$
and this Fourier symbol is given by
\begin{equation}\label{symbol}
\Theta^k_{\gamma}(\xi)=2^{2\gamma}\frac{\left|\Gamma\left(\tfrac{1}{2}+\tfrac{\gamma}{2}
+\tfrac{\sqrt{(\tfrac{n}{2}-1)^2-\mu_k}}{2}+\tfrac{\xi}{2}i\right)\right|^2}
{\left|\Gamma\left(\tfrac{1}{2}-\tfrac{\gamma}{2}+\tfrac{\sqrt{(\tfrac{n}{2}-1)^2-\mu_k}}{2}
+\tfrac{\xi}{2}i\right)\right|^2}.
\end{equation}
\end{theorem}
Since we are mainly interested in radial solutions $v=v(t)$, in many computations we will just need to consider the symbol for the first eigenspace $k=0$
 (that corresponds to the constant eigenfunctions):
\begin{equation*}\label{symbolk0}
\Theta^0_{\gamma}(\xi)=2^{2\gamma}\frac{\left|\Gamma(\tfrac{n}{4}+\tfrac{\gamma}{2}+\tfrac{\xi}{2}i)\right|^2}
{\left|\Gamma(\tfrac{n}{4}-\tfrac{\gamma}{2}+\tfrac{\xi}{2}i)\right|^2}.
\end{equation*}

 Now we look at the question of finding smooth solutions $v=v(t)$, $0<c_1\leq v\leq c_2$, for equation \eqref{equation2}, and we expect to have periodic solutions. The local case $\gamma=1$, presented below, provides some motivation for this statement and as we have mentioned, in the forthcoming paper \cite{paper2} we construct such periodic solutions from the variational point of view. Here we look at the geometrical interpretation of such solutions and provide a dynamical system approach for the problem.

 First, note that
 \begin{equation}\label{equation3}-\Delta w=c_{n,1}w^{\frac{n+2}{n-2}}, \quad w>0,\end{equation}
  is the constant scalar curvature equation for the metric $g_w$. It is well known  (\cite{Caffarelli-Gidas-Spruck,KMPS}) that positive solutions of equation \eqref{equation3} in $\mathbb R^n\backslash\{0\}$ must be radially symmetric and, if the singularity at the origin is not removable, then the solution must behave as in \eqref{asymptotics}.
After the change \eqref{wv}, our  equation \eqref{equation3} reduces to a standard second order ODE for the function $v=v(t)$:
 \begin{equation}\label{ODE1}
 -\ddot{v}+\tfrac{(n-2)^2}{4}v=\tfrac{(n-2)^2}{4}v^{\frac{n+2}{n-2}},\quad v>0.
 \end{equation}
 This equation it is easily integrated and the analysis of its phase portrait gives that all bounded solutions must be periodic (see, for instance, the lecture notes \cite{Schoen:notas}). More precisely, the Hamiltonian
 \begin{equation}\label{Hamiltonian1}H_1(v,\dot{v}):= \tfrac{1}{2}\dot v^2+\tfrac{(n-2)^2}{4}\left(\tfrac{(n-2)}{2n} v^{\frac{2n}{n-2}}-\tfrac{1}{2}v^2\right)\end{equation}
 is preserved along trajectories. Thus looking at its level sets we conclude that there exists a family of periodic solutions $\{v_L\}$ of periods $L\in(L_1,\infty)$. Here
 \begin{equation}\label{L1}
 L_1=\tfrac{2\pi}{\sqrt{n-2}}
 \end{equation}
 is the minimal period and it can be calculated from the linearization at the equilibrium solution $v_1\equiv 1$. These $\{v_L\}$ are known as the Fowler (\cite{Fowler}) or Delaunay solutions for the scalar curvature.

 Delaunay solutions are, originally,  rotationally symmetric surfaces with constant mean curvature and they have been known for a long time (\cite{Delaunay,Eells}). In addition, let $\Sigma\in\mathbb R^3$ be a noncompact embedded constant mean curvature surfaces with $k$ ends. It is known that any of such ends must be asymptotic to one of the Delaunay surfaces (\cite{Korevaar-Kusner-Solomon,Meeks:CCM}), which is very similar to what happens in the constant scalar curvature setting (see, for instance \cite{KMPS}), where any positive solution of the constant scalar curvature equation \eqref{equation3} must be asymptotic to a precise deformation of one of the $v_L$. Delaunay-type solutions have been also shown to exist for the first eigenvalue in Serrin's problem \cite{Sicbaldi}.

Let us comment here that the constant mean curvature case is very related to the constant $Q_{1/2}$ case. This is because, fixed $(X^{n+1},\bar g)$ a compact manifold with boundary $M^n$ such that there exists a defining function $\rho$ so that $g^+=\bar g/\rho^2$ is asymptotically hyperbolic and has constant scalar curvature $R_{g^+}=-n(n+1)$, then the conformal fractional Laplacian on $M$ with respect to the metric $g=\bar g|_{M}$ is given by $P^g_{1/2}=\frac{\partial}{\partial \nu}+\frac{n-1}{2} \mathcal H_g$, where $\frac{\partial}{\partial \nu}$ is the Neumann derivative for the harmonic extension and $\mathcal H_g$ the mean curvature of the boundary (\cite{Guillarmou-Guillope,MarChang}). Thus $Q_{1/2}^g$ coincides in this particular setting with $\mathcal H_g$ up to a multiplicative constant. However, there is a further restriction since in the present paper we consider only rotationally symmetric metrics on the boundary and thus, not allowing full generality for the original constant mean curvature problem.\\

Going back to \eqref{ODE1}, we would like to understand how much of this picture is preserved in the  non-local case, so we look for radial solutions of equation \eqref{equation0}, which becomes a fractional order ODE. On the one hand, we formulate the problem through the extension \eqref{equation1}. This point of view has the advantage that the new equation is local (and degenerate elliptic) but, on the other hand, it is a PDE with a non-linear boundary condition. Note that because we will be using the extension from Theorem \ref{equivaexten} for the calculation of the fractional Laplacian, we need to restrict ourselves to $\gamma\in(0,1)$ at this stage.

The first difficulty we encounter with our approach is how to write the original extension equation \eqref{equation1} in a natural way after the change of variables $r=e^{-t}$.
Looking at the construction of the conformal fractional Laplacian from the scattering equation \eqref{scattering-introduction} on a conformally compact Einstein manifold $(X^{n+1},g^+)$, we need to find a parametrization of hyperbolic space in such a way that its conformal infinity $M^n=\{\rho=0\}$ is precisely $(\mathbb R\times\mathbb S^{n-1}, g_0)$. The precise metric on the extension is $g^+=\bar g/\rho^2$
for
\begin{equation}\label{metric_introduction}
\bar{g}=d\rho^2+\left(1+\tfrac{\rho^2}{4}\right)^2dt^2+\left(1-\tfrac{\rho^2}{4}\right)^2 g_{\s^{n-1}},\end{equation}
where  $\rho\in(0,2)$ and $t\in\mathbb R$. The motivation for this change of metric will be made clear in Section \ref{relaclas}. In the language of Physics, $g^+$ is the Riemannian version of AdS space, a model solution of Einstein's equations which is important in the setting of the AdS/CFT correspondence since AdS space-time is the right background when studying thermodynamically stable black holes \cite{Hawking-Page,Witten}.

Rewriting the equations in this new metric, our original equation \eqref{equation1}, written in terms of the change \eqref{wv}, is equivalent to the  extension problem
\begin{equation}\label{fracyamextv}\left\{
\begin{split}
-\divergence_{\bar{g}}(\rho^a\nabla_{\bar{g}} V)+ E(\rho)V&=0\ \text{ in } (X^{n+1},\bar g),\\
V&=v \text{ on }\{\rho=0\},\\
-\tilde{d}_{\gamma}\lim_{\rho\rightarrow 0}\rho^{a}\partial_{\rho}V&=c_{n,\gamma}v^{\frac{n+2\gamma}{n-2\gamma}}\text{ on }\{\rho=0\},
\end{split}
\right.\end{equation}
where the expression for the lower order term $E(\rho)$ will be given in \eqref{Erho}. We look for solutions $V$ to \eqref{fracyamextv} that only depend on $\rho$ and $t$, and that are bounded between two positive constants.

We show first that equation \eqref{fracyamextv}  exhibits a Hamiltonian quantity that generalizes \eqref{Hamiltonian1}.

\begin{theorem}\label{thctthamiltonian}
Fix $\gamma\in(0,1)$. Let $V$ be a solution of \eqref{fracyamextv} only depending on $t$ and $\rho$. Then the Hamiltonian quantity
\begin{equation}\label{Hamiltonian}
H_\gamma(t):=\frac{1}{2}\int_{0}^{2}  \rho^{a}\left\{e_1(\rho)(\partial_{t}V)^2
-e(\rho)(\partial_\rho V)^2
-e_2(\rho)V^2\right\}\,d\rho
+C_{n,\gamma}v^{\frac{2n}{n-2\gamma}},
\end{equation}
is constant with respect to $t$. Here we write
\begin{equation}\label{e}
\begin{split}
&e_1(\rho)=\left( 1+\tfrac{\rho^2}{4}\right)^{-1}\left(1-\tfrac{\rho^2}{4}\right)^{n-1},\\
&e_2(\rho)=
\tfrac{n-1+a}{4}\left(1-\tfrac{\rho^2}{4}\right)^{n-2}\left(n-2+n\tfrac{\rho^2}{4}\right),\\
&e(\rho)=\left( 1+\tfrac{\rho^2}{4}\right)\left(1-\tfrac{\rho^2}{4}\right)^{n-1},
\end{split}\end{equation}
and the  constant is given by
\begin{equation}\label{constantC}
C_{n,\gamma}=\frac{n-2\gamma}{2n}\,\frac{c_{n,\gamma}}{\tilde{d}_{\gamma}}.
\end{equation}
\end{theorem}

 Hamiltonian quantities for fractional problems have been recently developed in the setting of one-dimensional solutions for fractional semilinear equations $(-\Delta)^{\gamma}w+f(w)=0$. The first reference we find a conserved Hamiltonian quantity for this type of non-local equations is the paper by Cabr{\'e} and Sol{\'a}-Morales \cite{CabreSola-Morales} for the particular case $\gamma=1/2$. The general case $\gamma\in(0,1)$ was carried out by Cabr{\'e} and Sire in \cite{paperxavi}. On the one hand, in these two papers \cite{paperxavi,CabreSola-Morales}, the authors impose a nonlinearity coming from a double-well potential
 and look for layers (i.e., solutions that are monotone and have prescribed limits at infinity), and they are able to write a Hamiltonian quantity that is preserved. In addition, if one considers the same problem but on hyperbolic space, one finds that the geometry at infinity plays a role and the analogous Hamiltonian is only monotone (see \cite{GonzalezSaezSire}).

  On the other hand, if, instead, one looks for radial solutions for semilinear equations, then
 Cabr\'e and Sire  in \cite{paperxavi} and Frank, Lenzman and Silvestre in \cite{FrankLenzmanSilvestre} have developed a monotonicity formula for the associated Hamiltonian. In the setting of radial solutions with an isolated singularity for the fractional Yamabe problem, our Theorem \ref{thctthamiltonian} states that, if one uses the metric \eqref{metric_introduction} to rewrite the problem, then the associated Hamiltonian \eqref{Hamiltonian} is constant along trajectories.\\

 If one insists on performing an ODE-type study for the PDE problem \eqref{fracyamextv}, a possibility is to plot a phase portrait of the boundary values (at $\rho=0$), while keeping in mind that the equation is defined on the whole extension. From this point of view, one can prove the existence of two critical points: the constant solutions $v_0\equiv 0$ and $v_1\equiv 1$. Moreover, there exists an explicit homoclinic solution $v_\infty$, whose precise expression will be given in \eqref{C}; it corresponds to the $n$-sphere \eqref{sphere}.

The next step is to linearize the equation. As we will observe in Section \ref{section:linear}, the classical Hardy inequality, rewritten in terms of the background metric $g_0$, decides the stability of the explicit solutions $v_0$, $v_1$ and $v_\infty$. Stability issues for semilinear fractional Laplacian equations have received a lot of attention recently.  Some references are: \cite{Cabre-Cinti} for the half-Laplacian, \cite{RosOton-Serra:extremal-solution} for extremal solutions with exponential nonlinearity, \cite{Fall:fractional-Hardy-potential} for semilinear equations with Hardy potential. In the particular case of  the fractional Lane-Emden equation, stability was considered in \cite{Davila-Dupaigne-Wei,Fazly-Wei:stable-solutions}, for $\gamma\in(0,1)$ and \cite{Fazly-Wey:higher} for $\gamma\in(1,2)$. We believe that our methods, although still at their initial stage, would provide tools for a unified approach for all $\gamma\in\left(0,\frac{n}{2}\right)$.\\

Finally, we consider the linearization of equation \eqref{equation2} around the equilibrium $v_1\equiv 1$:
\begin{equation*}
P_\gamma^{g_0} v=\tfrac{n+2\gamma}{n-2\gamma}v\quad \text{on}\quad \mathbb R\times \mathbb S^{n-1},
\end{equation*}
and look at the projection over each eigenspace $\langle E_k\rangle$, $k=0,1,\ldots$,
\begin{equation}\label{linearization-k}P_\gamma^k v_k=\tfrac{n+2\gamma}{n-2\gamma}v_k.\end{equation} Although we will not provide a complete calculation of the spectrum, we can say the following:

 \begin{theorem}\label{theorem:linearization}
For the projection $k=0$, equation \eqref{linearization-k} has periodic solutions $v(t)$ with period $L_\gamma=\frac{2\pi}{\sqrt{\lambda_\gamma}}$, where $\lambda_\gamma$  is the unique positive solution of \eqref{eqlambdak}. In addition,
$$\lim_{\gamma\to 1}L_\gamma =L_1,$$
so we recover the classical case \eqref{L1}.
 \end{theorem}
We conjecture that $L_{\gamma}$ is the minimal period of the radial periodic solutions for the nonlinear problem \eqref{equation2}.\\

\begin{remark}
We also give some motivation to show that the projection on the $k$-eigenspace of \eqref{linearization-k} does not have periodic solutions if $k=1,2,\ldots$.
\end{remark}

Theorem \ref{theorem:linearization} gives the existence of periodic radial solutions for the linear problem. In addition, the existence of a conserved Hamiltonian hints that the original non-linear problem has periodic solutions too. Moreover, the period $L_\gamma$ behaves well under the limit $\gamma\to 1$. Using the Hamiltonian and the methods in \cite{paperxavi} one can prove that periodic solutions also behave well under this limit. We leave this proof for the forthcoming paper \cite{paper2}.

The construction of Delaunay solutions allows for many further studies. For instance, as a consequence of their construction one obtains the non-uniqueness of the solutions for the fractional Yamabe problem in the positive curvature case, since it gives different conformal metrics on $\mathbb S^1(L)\times \mathbb S^{n-1}$ that have constant fractional curvature. This is well known in the scalar curvature case (see the lecture notes \cite{Schoen:notas} for an excellent review, or the paper \cite{Schoen:number}). In addition, this gives some motivation to define a total fractional scalar curvature functional, which maximizes the standard fractional Yamabe quotient (\cite{MarQing})  across conformal classes. We hope to return to this problem elsewhere.

From another point of view, Delaunay solutions can be used in gluing problems. Classical references are, for instance, \cite{Mazzeo-Pacard:isolated, Mazzeo-Pollack-Uhlenbeck} for the scalar curvature, and \cite{Mazzeo-Pacard:Delaunay-ends,Mazzeo-Pacard-Pollack} for the construction of constant mean curvature surfaces with Delaunay ends.\\

There is an alternative notion of fractional curvature and fractional perimeter defined from the singular integral definition of the fractional Laplacian which gives a different quantity than our $Q^g_\gamma$. \cite{Caffarelli-Roquejoffre-Savin} introduces the notion of nonlocal mean curvature for the boundary of a set in $\r^n$ (see also the review \cite{Valdinoci:review}), and it has also received a lot of attention recently. Finding Delaunay-type surfaces with constant nonlocal mean curvature has just been accomplished in \cite{Cabre-Fall-Sola-Weth}. For a related nonlocal equation, but different than nonlocal mean curvature, the recent paper  \cite{Davila-delPino-Dipierro-Valdinoci} establishes variationally the existence of Delaunay-type hypersurfaces.\\

We finally comment that the negative fractional curvature case has not been explored yet, except for the works \cite{Chen-Veron:1,Chen-Veron:2}. They consider singular solutions for the problem $(-\Delta)^\gamma w+|w|^{p-1}w=0$ in a domain $\Omega$ with zero Dirichlet condition on $\partial\Omega$. This setting is very different from the positive curvature case because the maximum principle is valid here. We also cite the work \cite{Quittner-Reichel}, where they consider singular solutions of $\Delta W=0$ in a domain $\Omega$ with a nonlinear Neumann boundary condition $\partial_\nu W=f(x,W)-W$ on $\partial\Omega$.\\

The paper will be structured as follows: in Section \ref{section:preliminaries} we
will recall some standard background on the fractional Yamabe problem. In particular we present the equivalent formulation as an extension problem coming from scattering theory. In section \ref{relaclas} we will give a geometric interpretation of the problem. Next, in section \ref{simbolo} we will analyze the scattering equation to give a proof for theorem \eqref{th1}. That is, we will compute the Fourier symbol for the conformal fractional Laplacian. In section \ref{section:ODE} we face the problem from an ODE-type point of view. This kind of study over the extension problem \eqref{fracyamextv} gives us two equilibria and the existence of a Hamiltonian quantity conserved along the trajectories. Moreover we will find in Section \ref{section:explicit} an explicit homoclinic solution, which corresponds to the $n-$sphere. Finally, in Section \ref{section:linear} we will perform a linear analysis close to the constant solutions which corresponds to the $n$-cylinder.

\section{Preliminaries}\label{section:preliminaries}

The conformal Laplacian operator for a Riemannian metric $g$ on a $n$-dimensional manifold $M$ is defined as
\begin{equation}\label{Conformallaplacian}
L_g=-\Delta_g+c_{n} R_g,\quad\text{where } c_{n}=\tfrac{(n-2)}{4(n-1)},
\end{equation}
and $R_g$ is the scalar curvature. The conformal Laplacian is a conformally covariant operator, indeed, given $g_w$ and $g$ two conformally related metrics with $g_w=w^{\frac{4}{n-2}}g$, then the operator $L_g$ satisfies %
\begin{equation*}
L_{g_w}(f)=w^{-\frac{n+2}{n-2}}L_g(wf),
\end{equation*}
for every $f\in \mathcal C^{\infty}(M)$.
In the case $f=1$ we obtain the classical scalar curvature equation
$$-\Delta_g w+c_nR_g w=c_nR_{g_w}w^{\tfrac{n+2}{n-2}},$$
which, in the flat case, is precisely equation \eqref{equation3}.\\

The fractional Laplacian on $\mathbb R^n$ is defined through Fourier transform as
$$\widehat{(-\Delta)^{\gamma}w}=|\xi|^{2 \gamma}\widehat{w},\quad  \forall\gamma\in \r.$$
Note that we use the Fourier transform defined by
\begin{equation*}
\widehat{w}(\xi)=(2\pi)^{-n/2}\int_{\r^n}w(x)e^{-i\xi\cdot x}\,dx.
\end{equation*}
Let
$\gamma\in (0,1)$ and $u\in L^{\infty}\cap \mathcal C^2$ in $\r^n$, the fractional Laplacian in $\r^n$ can also be defined by

\begin{equation*}\label{deffraclapla}
 (-\Delta)^{\gamma}w(x)=\kappa_{n,\gamma}\text{P.V. }\int_{\r^n}\frac{w(x+y)-w(x)}{|y|^{(n+2 \gamma)}}\,dy,
\end{equation*}
where $P.V. $ denotes the principal value, and the constant $\kappa_{n,\gamma}$ (see \cite{Landkof}) is given by
$$\kappa_{n,\gamma}=\pi^{-\tfrac{n}{2}}2^{2\gamma}
\tfrac{\Gamma\left(\tfrac{n}{2}+\gamma\right)}{\Gamma(1-\gamma)}\gamma.$$

Caffarelli-Silvestre introduced in \cite{CaffarelliSilvestre} a different way to compute the fractional Laplacian in $\r^n$ for $\gamma\in(0,1)$. Take coordinates $x\in\mathbb R^n$, $y\in\mathbb R_ +$. Let $w$ be any smooth function defined on $\r^n$ and consider the extension $W:\r^n\times \r^+\rightarrow\r$ solution of the following partial differential equation:
\begin{equation}\label{Caffarelli-Silvestre-extension}\left\{ \begin{split}
\divergence(y^a \nabla W)&=0,\ \ x\in\r^n,\ y\in\r,\\
W(x,0)&=w(x), \ \ x\in\r^n,
\end{split}\right.\end{equation}
where $a=1-2\gamma$.
Note that we can write $W=K_\gamma *_x w$, where $K_\gamma$ is the Poisson kernel
\begin{equation}\label{Poisson-kernel}
K_\gamma(x,y)=c\frac{y^{1-a}}{(|x|^2+y^2)^{\frac{n+1-a}{2}}},
\end{equation}
and $c=c(n,\gamma)$ is a multiplicative constant which is chosen so that, for all $y>0,$ $\int K_\gamma(x,y)\, dx=1$.
In addition,
$$(-\Delta)^{\gamma}w=-\tilde{d}_{\gamma}\lim_{y\rightarrow 0^+}y^a\partial_y W,$$
for the constant
\begin{equation}\label{dtg}
 \tilde{d}_{\gamma}=-\frac{2^{2\gamma-1}\Gamma(\gamma)}{\gamma\Gamma(-\gamma)}.
\end{equation}

One can generalize this construction to the curved setting. Let $X^{n+1}$ be a smooth manifold of dimension $n+1$ with smooth boundary $\partial X=M^n.$ A defining function in $\overline{X}$ for the boundary $M$ is a function $\rho$  which satisfies:
\begin{equation}\label{propdf}
\rho>0 \text{ in } X,\quad
\rho=0 \text{ on } M\quad \text{and} \quad
d\rho\neq 0 \text{ on } M.
\end{equation}
A Riemannian metric $g^+$ on $X$ is conformally compact if $(\overline{X},\bar{g})$ is a compact Riemannian manifold with boundary $M$ for a defining function $\rho$ and $\bar{g}=\rho^2g^+$.
Any conformally compact manifold $(X,g^+)$ carries a well-defined conformal structure $[g]$ on the boundary $M$, where $g$ is the restriction of $\bar{g}|_M$. We call $(M,[g])$ the conformal infinity of the manifold $X$.
We usually write these conformal changes on $M$ as $g_w=w^{\frac{4}{n-2\gamma}}g$ for a positive smooth function $w$.

A conformally compact manifold $(X,g^+)$ is called conformally compact Einstein manifold if, in addition, the metric satisfies the Einstein equation $Ric_{g^+}=-ng^+$, where $Ric$ represents the Ricci tensor.
One knows \cite{Graham} that
given a conformally compact Einstein manifold $(X,g^+)$ and a representative $g$ in $[g]$ on the conformal infinity $M^n$, there is an unique defining function $\rho$ such that one can find coordinates on a tubular neighborhood $M\times(0,\varepsilon)$ in $X$ in which $g^+$ has the normal form
\begin{equation}\label{confcomeinm}
g^+=\rho^{-2}(d\rho^2+g_\rho),
\end{equation}
where $g_\rho$ is a family on $M$ of metrics depending on the defining function and satisfying $g_{\rho}|_M=g$.%

Let $(X,g^+)$ be a conformally compact Einstein manifold with conformal infinity $(M,[g])$.
 It is well known from scattering theory \cite{GrahamZorski,MazzeoMelrose,Guillarmou} that, given $w\in \mathcal C^{\infty}(M)$ and $s\in \c $, if $ s(n-s) $ does not belong to the set of $L^2$-eigenvalues of $-\Delta_{g^+}$ then the eigenvalue problem
\begin{equation}\label{eigp}
-\Delta_{g^+}U-s(n-s)U=0 \text{ in } X,
\end{equation}
has a unique solution of the form
\begin{equation}\label{formau}
U=W \rho^{n-s}+ W_1 \rho^s, \quad W,W_1 \in \mathcal C^{\infty}(\overline{X}),\ \ W|_{\rho=0}=w.
\end{equation}
Taking a representative $g$ of the conformal infinity $(M,[g])$ we can define a family of meromorphic pseudo-differential operators $S(s)$, called scattering operators, as
\begin{equation}\label{scattering}
S(s)w=W_1|_M.
\end{equation}
The case that the order of the operator is an even integer was studied in \cite{GrahamZorski}. More precisely, suppose that $m\in\n$ and $m\leq\frac{n}{2}$ if $n$ is even, and that $(\frac{n}{2})^2-m^2$ is not an $L^2-$eigenvalue of $-\Delta_{g^+}$, then $S(s)$
has a simple pole at $s=\frac{n}{2}+m$. Moreover, if $P^{g}_m$ denotes the conformally invariant GJMS-operator on $M$ constructed in \cite{GrahamJenneMasonSparling} then
$$c_mP^{g}_m=-\residue_{s=\frac{n}{2+m}}S(s),\ \ \ c_m=(-1)^m[2^{2m}m!(m-1)!]^{-1},$$
where $\residue_{s=s_0}S(s)$ denotes the residue at $s_0$ of the meromorphic family of operators $S(s)$.
In particular,
   if $m=1$ we have the conformal Laplacian  $P_1^{g}=L_g$ from \eqref{Conformallaplacian}, and if $m=2$, the Paneitz operator $$P_2^{g}=(-\Delta_{g})^2+\delta(a_nR_{g}+b_nRic_{g})d+\frac{n-4}{2}Q_{g},$$
where $Q_g$ is the $Q$-curvature and $a_n,\ b_n$ are dimensional constants (\cite{Paneitz}).

It is also possible to define conformally covariant fractional powers of the Laplacian in the
case $\gamma\not\in\n$. For the rest of the paper, we set $\gamma\in\left(0,\frac{n}{2}\right)$ not an integer, $s=\frac{n}{2}+\gamma$. In addition, assume that $s(n-s)$ is not an $L^2$-eigenvalue for $-\Delta_{g^+}$ and that the first eigenvalue $\lambda_1(-\Delta_{g^+})>s(n-s)$.  In this setting:

\begin{definition}
 We define the conformally covariant fractional powers of the Laplacian as
\begin{equation}\label{pdg}
 P^g_{\gamma}[g^+,g]=d_{\gamma}S\left(\tfrac{n}{2}+\gamma\right),\text{ where }\ d_{\gamma}=2^{2\gamma}\frac{\Gamma(\gamma)}{\Gamma(-\gamma)}.
\end{equation}
\end{definition}
As a pseudodifferential operator, its principal symbol coincides with the one of $(-\Delta_{g})^{\gamma}$.
In the rest of the paper, once $g^+$ is fixed, we will use the simplified notation:
$$P_\gamma^{g}:=P_\gamma[g^+,g].$$
These operators satisfy the conformal property
\begin{equation}\label{conformproperty}
 P_{\gamma}^{g_{w}}f=w^{-\frac{n+2\gamma}{n-2\gamma}}P_{\gamma}^{g}(wf), \quad \forall f\in \mathcal C^{\infty}(M),
\end{equation}
for a change of metric $$g_{w}:=w^{\frac{4}{n-2\gamma}}g,\ w>0.$$

\begin{definition}
We define the fractional order curvature as:
$$Q_\gamma^{g}:=P_\gamma^{g}(1).$$
\end{definition}
Note that up to multiplicative constant, $Q_1$ is the classical scalar curvature and $Q_2$ is the so called $Q$-curvature.
\begin{remark}
Using the previous definition we can express the conformal property \eqref{conformproperty} as
\begin{equation}\label{confQ}
 P_{\gamma}^{g}(w)=w^{\frac{n+2\gamma}{n-2\gamma}}Q_{\gamma}^{g_{w}}.
\end{equation}
\end{remark}%

Explicit formulas for $P_{\gamma}^{g}$ are not known in general, however, Branson \cite{Brason}, gave an explicit formula for the conformal Laplacian on the standard sphere, i.e,
\begin{equation}\label{P_g^g1C}
P_{\gamma}^{g_{\s^n}}=
\frac{\Gamma(B+\gamma+\frac{1}{2})}{\Gamma(B-\gamma+\frac{1}{2})},
\end{equation}
where $B=\sqrt{-\Delta_{g_{\s^n}}+(\frac{n-1}{2})^2}$.
For example,%
$$P_{1}^{g_{\s^n}}=-\Delta_{g_{\s^n}}+\tfrac{n(n-2)}{4},\quad P_{1/2}^{g_{\s^n}}=\sqrt{-\Delta_{g_{\s^n}}+\left(\tfrac{n-1}{2}\right)^2}.$$

From \eqref{P_g^g1C} we can compute the fractional curvature on the unit sphere as
\begin{equation}\label{fqsphere}
Q_{\gamma}^{g_{\s^n}}=P^{g_{\s^n}}_{\gamma}(1)=\frac{\Gamma(\frac{n}{2}+\gamma)}{\Gamma(\frac{n}{2}-\gamma)}.
\end{equation}

It is proven in \cite{MarChang} (see also the more recent paper \cite{CaseChang}) that the conformal fractional Laplacian is the Dirichlet-to-Neumann operator for an extension problem that generalizes \eqref{Caffarelli-Silvestre-extension}:

\begin{theorem}\label{equivaexten}
 Let $\gamma\in(0,1)$ and $(X,g^+)$ be a conformally compact Einstein manifold with conformal infinity $(M,[g])$. For any defining function $\rho$ of $M$ satisfying \eqref{confcomeinm} in $X$, the scattering problem \eqref{eigp}-\eqref{formau}
is equivalent to
\begin{equation}\label{divE}\left\{
 \begin{split}
  -\divergence(\rho^a\nabla W)+E(\rho)W&=0\text{ in }(X,\bar{g}),\\
W&=w\text{ on } M,
 \end{split}\right.
\end{equation}
where
$$\bar{g}=\rho^2g^+,\ \ W=\rho^{s-n}U,\ \ s=\tfrac{n}{2}+\gamma,\ \ a=1-2\gamma.$$
and the derivatives in \eqref{divE} are taken respect to the metric $\bar{g}$. The lower order term is given by
\begin{equation*}\label{E1}
 E(\rho)=-\Delta_{\bar{g}}(\rho^{\frac{a}{2}})\rho^{\frac{a}{2}}
 +\left(\gamma^2-\tfrac{1}{4}\right)\rho^{-2+a}+\tfrac{n-1}{4n}R_{\bar{g}}\rho^a,
\end{equation*}
or written back in the metric $g^+$,
\begin{equation}\label{E}
E(\rho)=\rho^{-1-s}(-\Delta_{g^+}-s(n-s))\rho^{n-s}.
\end{equation}
In addition, we have the following formula for the calculation of the conformal fractional Laplacian
\begin{equation*}
P_{\gamma}^{g}w=-\tilde{d}_{\gamma}\lim_{\rho\rightarrow 0}\rho^{a}\partial_{\rho} W,
\end{equation*}
where $\tilde{d}_{\gamma}$ is defined as
\begin{equation*}
 \tilde{d}_{\gamma}=-\frac{2^{2\gamma-1}\Gamma(\gamma)}{\gamma\Gamma(-\gamma)}.
\end{equation*}
\end{theorem}
\begin{remark}\label{remark:Euclidean}
If $X$ is the hyperbolic space $\mathbb H^{n+1}$, identified with the upper half space $\r^{n+1}_+$ with the metric $g^+=\frac{dy^2+|dx|^2}{y^2}$, then the conformal infinity is simply $M=\r^n$ with the standard Euclidean metric $|dx|^2$ and  therefore, problem \eqref{divE} is precisely the extension problem considered by Caffarelli-Silvestre \eqref{Caffarelli-Silvestre-extension}. As a consequence, the conformal fractional Laplacian reduces to the standard fractional Laplacian without curvature terms, i.e., $P^{|dx|^2}_{\gamma}=(-\Delta)^{\gamma}$.
\end{remark}

Now we are going to choose a suitable defining function $\rho^*$, in order to transform the problem \eqref{confcomeinm} into one of pure divergence form. We follow \cite{MarChang,CaseChang}:
\begin{theorem}\label{E0}
 Set $\gamma\in(0,1)$. Let $(X,g^+)$ be a conformally compact Einstein manifold with conformal infinity $(M,[g])$, and such that $\lambda_1(-\Delta_{g^+})>\frac{n^2}{4}-\gamma^2$. Assuming that $\rho$ is a defining function satisfying \eqref{propdf}, there exists another (positive) defining function $\rho^*$ on $X$, %
such that for the term $E(\rho)$ defined in \eqref{E} we have $$E(\rho^*)=0.$$
The asymptotic expansion of this new defining function is
 \begin{equation*}
 \rho^*=\rho\left(1+\frac{2 Q_{\gamma}^{g}}{(n-2\gamma)d_{\gamma}}\rho^{2\gamma}+O(\rho^2)\right).
 \end{equation*}
In addition, the metric $g^*=(\rho^*)^2g^+$ satisfies $g^*|_{\rho=0}=g$ and has asymptotic expansion
$$g^*=(d\rho^*)^2[1+O((\rho^*)^{2\gamma})]+g[1+O((\rho^*)^{2\gamma})].$$

The scattering problem \eqref{formau}-\eqref{eigp} is  equivalent to the following one.
\begin{equation*}\left\{
\begin{split}
 -\divergence((\rho^*)^a\nabla W)&=0\text{ in }(X,g^*),\\
W&=w\text{ on }M,
\end{split}\right.
\end{equation*}
where the derivatives are taken with respect to the metric $g^*$ and $W=(\rho^*)^{s-n}U$.
Moreover
\begin{equation}\label{Pgrhostar}
 P_{\gamma}^{g}w=-\tilde{d}_{\gamma}\lim_{\rho^*\rightarrow 0}(\rho^*)^a\partial_{\rho^*}W+wQ_{\gamma}^{g}.
\end{equation}
\end{theorem}

The fractional Yamabe problem is, for $\gamma\in\left(0,\tfrac{n}{2}\right)$, to find a new metric $g_w=w^{\frac{4}{n-2\gamma}}g$  on $M$ conformal to $g$, with constant fractional curvature $Q_{\gamma}^{g_{w}}$.
Using the conformal property \eqref{confQ} the Yamabe problem is equivalent to find $w$ a strictly positive smooth function on $M$ satisfying
\begin{equation}\label{fracyamp}
P_{\gamma}^{g}(w)=cw^{\frac{n+2\gamma}{n-2\gamma}},\
\  w>0.
\end{equation}
In this paper we are interested in the positive curvature case, and the constant $c=c_{n,\gamma}$ will be normalized as in Proposition \ref{cte} below.

Thanks to Theorem \ref{equivaexten}, \eqref{fracyamp} is equivalent to the existence of a strictly positive $\mathcal C^\infty$ solution for extension problem:
\begin{equation}\label{fracyamext}\left\{
\begin{split}
-\divergence(\rho^a\nabla W)+E(\rho)W&=0\text{ in }(X,\bar{g}),\\
W&=w\text{ on }M,\\
-\tilde{d}_{\gamma}\lim_{\rho\rightarrow 0}\rho^ a \partial_{\rho}W &=c_{n,\gamma}w ^{\frac{n+2\gamma}{n-2\gamma}}\text{ on }M.
\end{split}\right.
\end{equation}
Using the special defining function from Theorem \ref{E0}, the fractional Yamabe problem \eqref{fracyamext} may be rewritten as
\begin{equation}\label{fracyamE0}\left\{
\begin{split}
-\divergence((\rho^*)^a\nabla W)&=0\text{ in }(X,g^*),\\
W&=w\text{ on }M,\\
-\tilde{d}_{\gamma}\lim_{\rho^*\rightarrow 0}(\rho^*)^ a \partial_{\rho^*}W +wQ^{g}_{\gamma}&=c_{n,\gamma}w ^{\frac{n+2\gamma}{n-2\gamma}}\text{ on }M.
\end{split}\right.
\end{equation}
Indeed we only need to rewrite the equation for the Yamabe problem \eqref{fracyamp} using the expression of $P_{\gamma}^{g}$ from \eqref{Pgrhostar}. Without danger of confusion, note that in general the solutions $W$ \eqref{fracyamext} and \eqref{fracyamE0} are different, but in the sequel they will be denoted by the same letter.

\begin{proposition}\label{cte}
The fractional curvature of the cylindrical metric $g_{w_1}={w_1}^{\frac{4}{n-2\gamma}}|dx|^2$ for
the conformal change
\begin{equation}\label{w1}w_1(x)=|x|^{-\frac{n-2\gamma}{2}},\end{equation} is the constant
 \begin{equation*}\label{cng}
c_{n,\gamma}=2^{2\gamma}\left(\frac{\Gamma(\frac{1}{2}(\frac{n}{2}+\gamma))}
{\Gamma(\frac{1}{2}(\frac{n}{2}-\gamma))}\right)^2>0.\end{equation*}
\end{proposition}
\begin{proof}
The value is calculated using the conformal property \eqref{confQ}, as follows:%
\begin{equation*}
Q_{\gamma}^{g_{w_1}}={w_1}^{-\frac{n+2\gamma}{n-2\gamma}}P_{\gamma}^{|dx|^2}(w_1)
={w_1}^{-\frac{n+2\gamma}
{n-2\gamma}}(-\Delta)^{\gamma}(w_1)=:c_{n,\gamma}.
\end{equation*}
The last equality follows from the calculation of the fractional Laplacian of a power function; it can be found in \cite{GonzalezMazzeoSire,xaviros}.
\end{proof}

\section{Geometric setting}\label{relaclas}

We give now the natural geometric interpretation of problem \eqref{equation0} and the extension formulation \eqref{equation1}. Thanks to Remark \ref{remark:Euclidean} and Theorem \ref{equivaexten}, the initial extension problem \eqref{equation1} is the scattering equation \eqref{eigp} in hyperbolic space, denoted by $X_1=\h^{n+1}$, with the metric $g^+=\frac{dy^2+|dx|^2}{y^2}$. Our point if view is to use the metric $g_0$ from \eqref{g0:introduction} as the representative of the conformal infinity, thus we need to rewrite the hyperbolic metric in a different normal form
\begin{equation}\label{normal-form}g^+=\frac{d\rho^2+g_\rho}{\rho^2} \quad\text{with}\quad g_\rho|_{\rho=0}=g_0,\end{equation}
for a suitable defining function $\rho$. Let us introduce some notation now. The conformal infinity (with an isolated singularity) is $M_1=\mathbb R^n\backslash \{0\}$,
which in polar coordinates can be represented as $M_1=\r^+\times\s^{n-1}$ and $|dx|^2=dr^2+r^2g_{\s^{n-1}}$,
or using this change of variable
$r=e^{-t}$, the Euclidean metric may be written as
\begin{equation}\label{relmetrclas}
 |dx|^2=e^{-2t}[dt^2+g_{\s^{n-1}}]=:e^{-2t}g_0.
\end{equation}
We consider now several models for hyperbolic space, identified with the Riemannian version of AdS space-time. These models are well known in cosmology since they provide the simplest background for the study of thermodynamically stable black holes (see \cite{Witten,Hawking-Page}, for instance, or the survey paper \cite{Chang-Qing-Yang}). Thus we write the hyperbolic metric as
\begin{equation}\label{HYmodel}
g^+=d\sigma^2+\cosh^2 \sigma \,dt^2+\sinh^2\sigma\,g_{\s^{n-1}},
\end{equation}
where $t\in\mathbb R$, $\sigma\in(0,\infty)$ $\theta\in\mathbb S^{n-1}$. %
Using the change of variable $R=\sinh \sigma$,
\begin{equation*}\label{metrR}
g^+=\frac{1}{1+R^2}\,dR^2+(1+R^2)\,dt^2+R^2g_{\s^{n-1}}.
\end{equation*}
This metric can be written in the normal form \eqref{normal-form} as
\begin{equation}\label{metrica21}
g^+=\rho^{-2}\left[d\rho^2+\left( 1+\tfrac{\rho^2}{4}\right)^2dt^2+\left( 1-\tfrac{\rho^2}{4}\right)^2 g_{\mathbb S^{n-1}}\right],
\end{equation}
for $\rho\in(0,2)$, $t\in \mathbb R$, $\theta\in\mathbb S^{n-1}$. Here we have used the relations
\begin{equation}\label{rhos}
\rho=2e^{-\sigma}\quad \text{and}\quad 1+R^2=\left( \tfrac{4-\rho^2}{4\rho} \right)^2
\end{equation}
Let $\bar g=\rho^2g^+$ be a compactification of $g^+$. Note that the apparent singularity at $\rho=2$ in the metric \eqref{metrica21} is just a consequence of the polar coordinate parametrization and thus the metric is smooth across this point.

We define now $X_2=(0,2)\times \s^1(L)\times\s^{n-1}$, with coordinates $\rho\in(0,2),\ t\in\s^1(L),\ \theta\in \mathbb S^{n-1}$,
and the same metric given by \eqref{metrica21}.
The conformal infinity $\{\rho=0\}$ is $M_2=\s^1(L)\times \s^{n-1}$,  with the metric given by $g_0=dt^2+g_{\s^{n-1}}.$

 Note that $(X_1,g^+_{\h^{n+1}})$ is a covering of $(X_2,g^+)$. Indeed, $X_2$ is the quotient $X_2=\h^{n+1}/\z\approx \r^n\times\s^1(L)$ with $\z$ the group generated by the translations, if we make the $t$ variable periodic. As a consequence, also $(M_1,|dx|^2)$ is a covering of $(M_2,g_0)$ after the  conformal change \eqref{relmetrclas}.\\

Summarizing, we denote $X=(0,2)\times\mathbb R\times\s^{n-1}$ and $M=\mathbb R\times\s^{n-1}$ and recall that the metric $\bar{g}=\rho^2g^+$ is given by
\begin{equation}\label{gbarfrac}
 \bar{g}=d\rho^2+\left(1+\tfrac{\rho^2}{4}\right)^2dt^2+\left(1-\tfrac{\rho^2}{4}\right)^2 g_{\s^{n-1}},\quad\text{and}\quad g_0=\bar{g}|_M=dt^2+g_{\s^{n-1}}.
\end{equation}
Equality \eqref{relmetrclas} shows that the metrics $|dx|^2$ and $g_0$ are conformally related and therefore using \eqref{wv},
we can write any conformal change of metric on $M$ as
\begin{equation}\label{hatEw}
g_v:=w^{\frac{4}{n-2\gamma}}|dx|^2=
v^{\frac{4}{n-2\gamma}}g_0.
\end{equation}
Our aim is to to find radial (in the variable $|x|$), positive solutions for \eqref{equation1} with an isolated singularity at the origin. Using $g_0$ as background metric on $M$, and writing the conformal change of metric in terms of $v$ as \eqref{hatEw},
this is equivalent to look for positive solutions $v=v(t)$ to \eqref{equation2} with $0<c_1\leq v\leq c_2$, and we hope to find those that are periodic in $t$.\\

Finally, we check that the background metric $g_0$ given in \eqref{gbarfrac} has constant
fractional curvature $Q^{g_0}_{\gamma}\equiv c_{n,\gamma}$. This is true because of the definition of $c_{n,\gamma}$ given in Proposition \ref{cte}, and the conformal equivalence given in \eqref{relmetrclas}. Thus, by construction, the trivial change $v_1\equiv 1$ is a solution to \eqref{equation2}.

\section{The conformal fractional Laplacian on $\r\times\s^{n-1}$.}\label{simbolo}

In this section we present the proof of Theorem \ref{th1}, i.e, the calculation of the Fourier symbol for the conformal fractional Laplacian on $\r\times\s^{n-1}$. This computation is based on the analysis of the scattering equation given in \eqref{eigp}-\eqref{formau} for the extension metric \eqref{metrica21}. We recall  that the scattering operator is defined as
\begin{equation}\label{scattering2}
P_\gamma^{g}w=S(s)w=W_1|_{\rho=0},
\end{equation}
and $s=\frac{n}{2}+\gamma$. The main step in the proof is to reduce \eqref{eigp} to an ODE that can be explicitly solved. Note that this idea of studying the scattering problem on certain Lorentzian models has been long used in Physics papers, but in general it is very hard to obtain explicit expressions for the solution and the majority of the existing results are numeric (see, for example, \cite{Fisicos}).

For the calculations below it is better to use the hyperbolic metric given in the coordinates \eqref{HYmodel}. Then the conformal infinity  corresponds to the value $\{\sigma=+\infty\}$.
 The scattering equation \eqref{eigp} can be written in terms of the variables $\sigma\in(0,\infty)$, $t\in\r$ and  $\theta\in\s^{n-1}$ as
\begin{equation}\label{eqs}
\partial_{{\sigma}{\sigma}}U+Q(\sigma)\partial_{\sigma}U+
\cosh^{-2}({\sigma})\partial_{tt}U+\sinh^{-2}(\sigma)\Delta_{\s^{n-1}}U+\left(\tfrac{n^2}{4}-\gamma^2\right)U=0,
\end{equation}
where $U=U({\sigma},t,\theta)$, and $$Q(\sigma)=\frac{\partial_{\sigma}(\cosh{\sigma}\sinh^{n-1}{\sigma})}{\cosh{\sigma}\sinh^{n-1}{\sigma}}.$$
With the change of variable \begin{equation}\label{cambioz}
z=\tanh(\sigma),\end{equation}
 equation \eqref{eqs} reads:
\begin{equation}\label{eqz}
\begin{split}
(1-z^2)^2\partial_{zz}U+\left(\tfrac{n-1}{z}-z\right)(1-z^2)\partial_z U+(1-z^2)\partial_{tt}U&
\\+\left(\tfrac{1}{z^2}-1\right)\Delta_{\s^{n-1}}U+\left(\tfrac{n^2}{4}-\gamma^2\right)U&=0.
\end{split}
\end{equation}
We compute the projection of equation \eqref{eqz} over each eigenspace of $\Delta_{\s^{n-1}}$. Given $k\in\n$, let $U_k(z,t)$ be the projection of $U$ over the eigenspace $\langle E_k\rangle$ associated to the eigenvalue $\mu_k=-k(k+n-2)$. Each $U_k$ satisfies the following equation:
\begin{equation}\label{equk}
(1-z^2)\partial_{zz}U_k+\left(\tfrac{n-1}{z}-z\right)\partial_z U_k+\partial_{tt}U_k+\mu_k\tfrac{1}{z^2}U_k
+\tfrac{\tfrac{n^2}{4}-\gamma^2}{1-z^2}U_k=0.
\end{equation}
Taking the Fourier transform \eqref{fourier} in the variable $t$ we obtain
\begin{equation}\label{equkfou}
(1-z^2)\partial_{zz}\widehat{U_k}+\left(\tfrac{n-1}{z}-z\right)\partial_z \widehat{U_k}
+\left[\mu_k\tfrac{1}{z^2}+\tfrac{\tfrac{n^2}{4}-\gamma^2}{1-z^2}-\xi^2\right]\widehat{U_k}=0.
\end{equation}
Fixed $k$ and $\xi$,
 we know that
\begin{equation}\label{u_k}
\widehat{U_k}=\widehat{w_k}(\xi)\varphi_k^{\xi}(z),
\end{equation}
where $\varphi:=\varphi_k^{\xi}(z)$ is the solution of the following ODE problem:
\begin{equation}\label{problemphi}\left\{
\begin{split}
&(1-z^2)\partial_{zz}\varphi+
\left(\tfrac{n-1}{z}-z\right)\partial_z\varphi+\left(\tfrac{\mu_k}{z^2}+\tfrac{\frac{n^2}{4}-\gamma^2}
{1-z^2}-\xi^2\right)\varphi =0,\\
&\text{has the expansion \eqref{formau} with }w\equiv 1\text{ near the conformal infinity }z=1,  \\
&\varphi \text{ is regular at }z=0.
\end{split}\right.
\end{equation}
 This ODE has only regular singular points $z$. The first equation in \eqref{problemphi} can be explicitly solved,
\begin{equation}\label{varphi1}
\begin{split}
\varphi(z)=&A(1-z^2)^{\frac{n}{4}-\frac{\gamma}{2}}z^{1-\frac{n}{2}+\sqrt{(\tfrac{n}{2}-1)^2-\mu_k}}
\Hyperg(a,b;c;z^2)\\
+&B(1-z^2)^{\frac{n}{4}-\frac{\gamma}{2}}z^{1-\frac{n}{2}-\sqrt{(\tfrac{n}{2}-1)^2-\mu_k}}\Hyperg(\tilde{a},\tilde{b};\tilde{c};,z^2),
\end{split}
\end{equation}
for any real constants $A,B$, where \begin{itemize}
        \item $a=\tfrac{-\gamma}{2}+\tfrac{1}{2}+\tfrac{\sqrt{(\tfrac{n}{2}-1)^2-\mu_k}}{2}+i\tfrac{\xi}{2}$,
        \item $b=\tfrac{-\gamma}{2}+\tfrac{1}{2}+\tfrac{\sqrt{(\tfrac{n}{2}-1)^2-\mu_k}}{2}-i\tfrac{\xi}{2}$,
        \item $c=1+\sqrt{(\tfrac{n}{2}-1)^2-\mu_k}$,
        \item $\tilde{a}=\tfrac{-\gamma}{2}+\tfrac{1}{2}-\tfrac{\sqrt{(\tfrac{n}{2}-1)^2-\mu_k}}{2}+i\tfrac{\xi}{2}$,
        \item $\tilde{b}=\tfrac{-\gamma}{2}+\tfrac{1}{2}-\tfrac{\sqrt{(\tfrac{n}{2}-1)^2-\mu_k}}{2}-i\tfrac{\xi}{2}$,
        \item $\tilde{c}=1-\sqrt{(\tfrac{n}{2}-1)^2-\mu_k}$,

      \end{itemize}
and $\Hyperg$ denotes the standard hypergeometric function introduced in Lemma \ref{propiedadeshiperg}. Note that we can write $\xi$ instead of $|\xi|$ in the arguments of the hypergeometric functions because $a=\bar{b}$, $\tilde{a}=\overline{\tilde{b}}$ and property \eqref{prop5} given in Lemma \ref{propiedadeshiperg}.

The regularity at the origin $z=0$ implies $B=0$ in \eqref{varphi1}.
Moreover, property \eqref{prop4} from Lemma \ref{propiedadeshiperg} makes it possible to rewrite $\varphi$ as
\begin{equation}\label{varphiz}
\begin{split}
\varphi(z)
=A&\left[\alpha(1+z)^{\frac{n}{4}-\frac{\gamma}{2}}(1-z)^{\frac{n}{4}-\frac{\gamma}{2}}z^{1-\frac{n}{2}
+\sqrt{(\tfrac{n}{2}-1)^2-\mu_k}} \right.\\ &
\cdot\Hyperg(a,b;a+b-c+1;1-z^2)\\
&+\beta (1+z)^{\frac{n}{4}+\frac{\gamma}{2}}(1-z)^{\frac{n}{4}+\frac{\gamma}{2}}z^{1-\frac{n}{2}
+\sqrt{(\tfrac{n}{2}-1)^2-\mu_k}}\\&
\left.\cdot\Hyperg(c-a,c-b;c-a-b+1;1-z^2)\right],
\end{split}
\end{equation}
where \begin{align}
\label{alpha}&\alpha=\tfrac{\Gamma\left(1+\sqrt{(\tfrac{n}{2}-1)^2-\mu_k}\right)\Gamma(\gamma)}
{\Gamma\left(\tfrac{1}{2}
+\tfrac{\gamma}{2}+\tfrac{\sqrt{(\tfrac{n}{2}-1)^2-\mu_k}}{2}-i\tfrac{\xi}{2}\right)\Gamma\left(\tfrac{1}{2}
+\tfrac{\gamma}{2}+\tfrac{\sqrt{(\tfrac{n}{2}-1)^2-\mu_k}}{2}+i\tfrac{\xi}{2}\right)},\\
\notag&\beta=\tfrac{\Gamma\left(1+\sqrt{(\tfrac{n}{2}-1)^2-\mu_k}\right)
\Gamma(-\gamma)}{\Gamma\left(\tfrac{-\gamma}{2}+\tfrac{1}{2}+\tfrac{\sqrt{(\tfrac{n}{2}-1)^2-\mu_k}}{2}
+i\tfrac{\xi}{2}\right)\Gamma\left(\tfrac{-\gamma}{2}+\tfrac{1}{2}+\tfrac{\sqrt{(\tfrac{n}{2}-1)^2-\mu_k}}{2}
-i\tfrac{\xi}{2}\right)}.
\end{align}
The constant coefficient $A$ will be fixed from the second statement in \eqref{problemphi}. From the definition of the scattering operator in \eqref{scattering2}, $\varphi$ must have the asymptotic expansion near $\rho=0$
\begin{equation}\label{expansionu}
\varphi(\rho)=\rho^{n-s}(1+...)+\rho^s(\widehat{S^k}(s)1+...),
\end{equation}
where $S^k(s)$ is the projection of the scattering operator $S(s)$ over the eigenspace $\langle E_k\rangle$.

We now use the changes of variable \eqref{cambioz} and \eqref{rhos}, obtaining
\begin{equation}\label{zrho}
z=\tanh(\sigma)=\frac{4-\rho^2}{4+\rho^2}=1-\frac{1}{2}\rho^2+\cdots .
\end{equation}
Therefore, substituting \eqref{zrho} into \eqref{varphiz} we can express $\varphi$ as a function on $\rho$ as follows
\begin{equation*}
\begin{split}
\varphi(\rho)\sim
A&\left[\alpha \rho^{\frac{n}{2}-\gamma}\Hyperg(a,b;a+b-c+1;\rho^2)\right.\\
&\,+\left.\beta \rho^{\frac{n}{2}+\gamma}\Hyperg(c-a,c-b;c-a-b+1;\rho^2)\right],\quad \text{as }\rho\to 0.
\end{split}
\end{equation*}
Using property \eqref{prop2} from Lemma \ref{propiedadeshiperg} below, we have that near the conformal infinity,
\begin{equation}\label{varphirhocero}
\begin{split}
\varphi(\rho)\simeq
 A\left[\alpha\rho^{\frac{n}{2}-\gamma}+\beta\rho^{\frac{n}{2}+\gamma}+\ldots\right] .
\end{split}
\end{equation}
Comparing \eqref{varphirhocero} with the expansion of $\varphi$ given in \eqref{expansionu}, we have  \begin{equation}\label{A}
A=\alpha^{-1},
\end{equation}
and
\begin{equation}\label{S(s)}
\widehat{S^k}(s)=%
\beta \alpha^{-1}.
\end{equation}
Recalling the definition of the conformal fractional Laplacian given in \eqref{pdg}, and taking into account \eqref{u_k}, we can assert that the Fourier symbol $\Theta^k_{\gamma}(\xi)$ for the projection $P_\gamma^k$ of the conformal fractional Laplacian $P_{\gamma}^{g_0}$ satisfies
\begin{equation*}\label{relsymbolS}
\Theta^k_{\gamma}(\xi)=\frac{\Gamma(\gamma)}{\Gamma(-\gamma)}2^{2\gamma}\widehat{S^k}(s).
\end{equation*}
From here we can calculate the value of this symbol and obtain \eqref{symbol}; just take \eqref{S(s)} into account and property \eqref{prop1g} from Lemma \ref{propiedadesgamma}.
This completes the proof of Theorem \ref{th1}.\\

\begin{remark}
When $\gamma=m$, an integer, we recover the principal symbol for the GJMS operators $P^{g_0}_m$. Indeed, from Theorem \ref{th1} we have that for any dimension $n>2m$, the Fourier symbol of $P^{g_0}_m$ is given by
\begin{equation*}
\begin{split}
\Theta^k_m(\xi)&=2^{2m}\frac{|\Gamma(\tfrac{1}{2}+\tfrac{m}{2}+\tfrac{\sqrt{(\tfrac{n}{2}-1)^2-\mu_k}}{2}
+\tfrac{\xi}{2}i)|^2}
{|\Gamma(\tfrac{1}{2}-\tfrac{m}{2}+\tfrac{\sqrt{(\tfrac{n}{2}-1)^2-\mu_k}}{2}+\tfrac{\xi}{2}i)|^2}\\&=
2^{2m}\prod_{j=1}^m\left(\tfrac{[4(m-j)-m+1+\sqrt{(\tfrac{n}{2}-1)^2+k(k+n-1)}]^2}{4}+\tfrac{\xi^2}{4}\right)\\
&=\Psi(m,n,k,\xi,\xi^2,...,\xi^{2m-1})+\xi^{2m},
\end{split}
\end{equation*}
where we have used the property \eqref{prop2g} of the Gamma function given in Lemma \ref{propiedadesgamma}. Note that $\Psi$ is a polynomial function on $\xi$ of degree less than $2m$.

For instance, for the classical case $m=1$,
$$\Theta^k_1(\xi)=\xi^2+\tfrac{(n-2)^2}{4}-\mu_k,\quad k=0,1,\ldots,$$
so we recover the usual conformal Laplacian $P^{g_0}_1$ given by
\begin{equation*}
P^k_1(v)=-\ddot{v}+\left[\tfrac{(n-2)^2}{4}-\mu_k \right]v, \quad k=0,1,\ldots,\end{equation*}
that is the operator appearing in \eqref{ODE1} when applied to radial functions.
\end{remark}

This proof also allows us to explicitly calculate the special defining function $\rho^*$ from Theorem \ref{E0}:
\begin{corollary}\label{corollary:rho*}
We have
$$(\rho^*)^{n-s}=\alpha^{-1} \left(\tfrac{4\rho}{4+\rho^2}\right)^{\frac{n}{2}-{\gamma}}
\Hyperg\left(\tfrac{n}{4}-\tfrac{\gamma}{2},\tfrac{n}{4}-\tfrac{\gamma}{2};
\tfrac{n}{2},\left(\tfrac{4-\rho^2}{4+\rho^2}\right)^2\right),$$
where $\alpha$ is the constant from \eqref{alpha}.
As a consequence, $\rho^*\in(0,\rho_0^*)$ where we have defined $(\rho_0^*)^{n-s}=\alpha^{-1}$.
\end{corollary}

\begin{proof}
From the proof of Theorem \ref{E0}, which corresponds to Lemma 4.5 in \cite{MarChang}, we know that
$$\rho^*=(\varphi_0^0)^{\frac{1}{n-s}}(z),$$
where $\varphi$ is the solution of \eqref{problemphi}. Thus from formula \eqref{varphi1} for $B=0$ and the relation between $z$ and $\rho$ from \eqref{zrho} we arrive at the desired conclusion. The behavior when $\rho\to 2$ can be calculated directly from \eqref{varphi1} and, as a consequence, $(\rho_0^*)^{n-s}=\varphi(0)=\alpha^{-1}$.
\end{proof}

\begin{lemma}\label{propiedadeshiperg}
 \textnormal{ \cite{Abramowitz,SlavyanovWolfganglay}} Let $z\in\c$. The hypergeometric function is defined for $|z| < 1$ by the power series
$$ \Hyperg(a,b;c;z) = \sum_{n=0}^\infty \frac{(a)_n (b)_n}{(c)_n} \frac{z^n}{n!}=\frac{\Gamma(c)}{\Gamma(a)\Gamma(b)}\sum_{n=0}^\infty \frac{\Gamma(a+n)\Gamma(b+n)}{\Gamma(c+n)} \frac{z^n}{n!}.$$
It is undefined (or infinite) if $c$ equals a non-positive integer.
Some properties are
\begin{enumerate}
  \item The hypergeometric function evaluated at $z=0$ satisfies
  \begin{equation}\label{prop2}
  \Hyperg(a+j,b-j;c;0)=1; \  j=\pm1,\pm2,...
  \end{equation}
     \item If $|arg(1-z)|<\pi$, then
\begin{equation}\label{prop4}
  \begin{split}
  \Hyperg&(a,b;c;z)=
                     \frac{\Gamma(c)\Gamma(c-a-b)}{\Gamma(c-a)\Gamma(c-b)}
                     \Hyperg\left(a,b;a+b-c+1;1-z\right)
                      \\
                     +&(1-z)^{c-a-b}\frac{\Gamma(c)\Gamma(a+b-c)}
                     {\Gamma(a)\Gamma(b)}\Hyperg(c-a,c-b;c-a-b+1;1-z).
     \end{split} \end{equation}
\item The hypergeometric function is symmetric with respect to first and second arguments, i.e
\begin{equation}\label{prop5}
  \Hyperg(a,b;c;z)= \Hyperg(b,a;c;z).
  \end{equation}

\end{enumerate}
\end{lemma}

\begin{lemma}\label{propiedadesgamma}
 \textnormal{ \cite{Abramowitz,SlavyanovWolfganglay}} Let $z\in\c$. Some properties of the Gamma function $\Gamma(z)$ are
\begin{align}
      &\Gamma(\bar{z})=\overline{\Gamma(z)},\label{prop1g}\\
  &\Gamma(z+1)=z\Gamma(z), \label{prop2g}\\
   &\Gamma(z)\Gamma\left(z+\tfrac{1}{2}\right)=2^{1-2z}\sqrt{\pi}\,\Gamma(2z).  \label{prop3g}
   \end{align}
 Let $\psi(z)$ denote the Digamma function defined by $$\psi(z)=\frac{d\ln\Gamma (z)}{dz}=\frac{\Gamma'(z)}{\Gamma(z)}.$$
 This function has the expansion
\begin{equation}\label{propdg}
\psi(z)=\psi(1)+\sum_{m=0}^{\infty}\left(\tfrac{1}{m+1}-\tfrac{1}{m+z}\right).
\end{equation}
Let $B(z_1,z_2)$ denote the Beta function defined by
\begin{equation*}
B(z_1,z_2)=\frac{\Gamma(z_1)\Gamma(z_2)}{\Gamma(z_1+z_2)}.
\end{equation*}
If $z_2$ is a fixed number and $z_1>0$ is big enough, then this function behaves
\begin{equation}\label{propbeta}
B(z_1,z_2)\sim \Gamma(z_2)(z_1)^{-z_2}.
\end{equation}
\end{lemma}

We end this section with a remark on the classical fractional Hardy inequality. On Euclidean space $(\mathbb R^n, |dx|^2)$, it is well known  that, for all $w\in\mathcal C_0^\infty(\mathbb R^n)$ and $\gamma\in\left(0,\frac{n}{2}\right)$,
\begin{equation}\label{fractional-Hardy}
\begin{split}c_H\int_{\mathbb R^n}\frac{|w|^2}{|x|^{2\gamma}}\,dx&\leq
\int_{\mathbb R^n}|\xi|^{2\gamma} |\hat w(\xi)|^2\,d\xi
\\&=\int_{\mathbb R^n} |(-\Delta)^{\frac{\gamma}{2}}w|^2\,dx=\int_{\mathbb R^n} w(-\Delta)^\gamma w\,dx.\end{split}\end{equation}
Moreover, the constant $c_H$ is sharp (although it is not achieved) and its value is given by
$$c_H=c_{n,\gamma},$$
which is the constant in Proposition \ref{cte}. This is not a coincidence, since the functions that are used in the proof of the sharpness statement are suitable approximations of \eqref{w1}. This constant was first calculated in \cite{Herbst}, but there have been many references \cite{Yafaev,Beckner:Pitts,Frank-Lieb-Seiringer:Hardy}, for instance.

A natural geometric context for the fractional Hardy inequality is obtained by  taking $g_0$ as a background metric, and using the changes  \eqref{wv} and \eqref{g0:introduction}. Indeed, using the conformal relation given by expression \eqref{relation-introduction}, we conclude that \eqref{fractional-Hardy} is equivalent to the following:
\begin{equation}\label{Hardy-v}
c_{n,\gamma}\int_{\mathbb R\times \mathbb S^{n-1}} v^2\dvol_{g_0}\leq \int_{\mathbb R\times \mathbb S^{n-1}} v (P_\gamma^{g_0} v)\dvol_{g_0},
\end{equation}
for every $v\in\mathcal C^{\infty}_0(\mathbb R\times \mathbb S^{n-1})$.

\section{ODE-type analysis}\label{section:ODE}
In this section we fix $\gamma\in(0,1)$. As we have explained, the fractional Yamabe problem with an isolated singularity at the origin is equivalent to the extension problem \eqref{equation1}. We look for radial solutions of the form \eqref{wv}. Based on our previous study, it is equivalent to consider solutions $v=v(t)$ of the extension problem \eqref{fracyamextv}, for the metric \eqref{metric_introduction}. In this section we perform an ODE-type analysis for the PDE problem \eqref{fracyamextv}.

Firstly we calculate
\begin{equation}\label{divergence-barg}
\begin{split}
 \divergence_{\bar{g}}(\rho^{a}\nabla_{\bar{g}} V)=&\sum_{i,j}\frac{1}{\sqrt{|\bar{g}|}}\partial_{i}(\bar{g}^{ij}\rho^{a}\sqrt{|\bar{g}|}\partial_jV)\\
=& \tfrac{1}{e(\rho)}\partial_{\rho}\left(\rho^{a}e(\rho)\partial_{\rho}V\right)+\tfrac{\rho^{a}}
{(1+\frac{\rho^2}{4})^2}\partial_{tt}V+\tfrac{\rho^{a}}{(1-\frac{\rho^2}{4})^2}\Delta_{\s^{n-1}}V,
\end{split}
\end{equation}
where \begin{equation*}
e(\rho)=\left(1+\tfrac{\rho^2}{4}\right)\left(1-\tfrac{\rho^2}{4}\right)^{n-1}.
\end{equation*}
Using the expression given in \eqref{E},
\begin{equation}\label{Erho}
E(\rho)=\tfrac{n-1+a}{4}\rho^a\frac{n-2+n\frac{\rho^2}{4}}{\left(1+\frac{\rho^2}{4}\right)\left( 1-\frac{\rho^2}{4}\right)}.
\end{equation}

\begin{remark}\label{solorho}
Let $V$ be the (unique) solution of \eqref{fracyamextv}. If $v$ does not depend on the spherical variable $\theta\in \mathbb S^{n-1}$, then $V$ does not either. Analogously, if $v$ is independent on $t$ and $\theta$, then $V$ is just a function of $\rho$. The proof is a straightforward computation using that the variables in \eqref{divergence-barg} are separated.
\end{remark}

 As a consequence of the previous remark, it is natural to look for solutions $V$ of \eqref{fracyamextv} that only depend on $\rho$ and $t$, i.e. solutions of
\begin{equation}\label{problemaradial}\left\{
\begin{split}
-\frac{1}{e(\rho)}\partial_{\rho}\left(\rho^{a}(e(\rho)\partial_{\rho}V\right)-
\frac{\rho^{a}}{(1+\frac{\rho^2}{4})^2}\partial_{tt}V+E(\rho)V&=0\ \text{ for } \rho\in(0,2),t\in\r,\\
V&=v\ \text{ on }\{\rho=0\},\\
-\tilde{d}_{\gamma}\lim_{\rho\rightarrow 0}\rho^{a}\partial_{\rho}V&=c_{n,\gamma}v^{\frac{n+2\gamma}{n-2\gamma}}\ \text{ on }\{\rho=0\}.
\end{split}
\right.\end{equation}
Now we take the special defining function $\rho^*$ given in Theorem \ref{E0}, whose explicit expression is given in Corollary \ref{corollary:rho*}. Then we can rewrite the original problem \eqref{fracyamextv} in $g^*$, defined on the extension $X^*=\{\rho\in(0,\rho_0^*),t\in\mathbb R,\theta\in \mathbb S^{n-1}\}$, as
\begin{equation}\label{Yamfracspec}\left\{
\begin{split}
-\divergence_{g^*}((\rho^*)^a\nabla_{g^*}V)&=0\ \text{ in }(X^*,g^*),\\
V&=v\ \text{ on }\{\rho^*=0\},\\
-\tilde{d}_{\gamma}\lim_{\rho^*\to 0}(\rho^*)^a\partial_{\rho^*}V +c_{n,\gamma}v&=c_{n,\gamma}v^{\frac{n+2\gamma}{n-2\gamma}}\ \text{ on }\{\rho^*=0\},
\end{split}\right.
\end{equation}
where $g^*=\frac{(\rho^*)^2}{\rho^2}\bar{g}$, for $\rho^*=\rho^*(\rho)$.

 Note that Proposition \ref{cte} calculates the value $Q_\gamma^{g_0}\equiv c_{n,\gamma}$. The advantage of \eqref{Yamfracspec} over the original \eqref{fracyamextv} is that it is a pure divergence elliptic problem and has nicer analytical properties.

Next, if we look for radial solutions (that depend only on $t$ and $\rho^*$), then the extension problem \eqref{Yamfracspec} reduces to:
\begin{equation}\label{eqfracesp}\left\{
\begin{split}
\frac{1}{e^*(\rho)}\partial_{\rho^*}\left((\rho^*)^{a}e^*(\rho)\partial_{\rho^*}V\right)
+\frac{(\rho^*)^{a}}
{\left(1+\tfrac{\rho^2}{4}\right)^{2}}\,\partial_{tt}V&=0 \text{ for }t\in\r,\rho^*\in(0,\rho^*_0),\\
v&=V\text{ on }\{\rho^*=0\},\\
-\tilde{d}_{\gamma}(\rho^*)^a\partial_{\rho^*}V  +c_{n,\gamma}v&=c_{n,\gamma}v^{\frac{n+2\gamma}{n-2\gamma}}\text{ on }\{\rho^*=0\},
\end{split}\right.
\end{equation}
where
$$e^*(\rho)=\left(\tfrac{\rho^*}{\rho}\right)^2 e(\rho).$$

Summarizing, we will concentrate in problems \eqref{problemaradial} and \eqref{eqfracesp}. In some sense \eqref{eqfracesp} is closer to the local equation \eqref{ODE1} and shares many of its properties. For instance, it has two critical points: $v_0\equiv0$ and $v_1\equiv1$, since these are the only constant solutions of the boundary condition $v=v^{\frac{n+2\gamma}{n-2\gamma}}$ on $\rho^*=0$. Moreover, by uniqueness of the solution and Remark \ref{solorho}, the only critical points in the extension are simply $V_0\equiv 0$ and $V_1\equiv 1$.

\begin{remark}
The calculation of the critical points $v_0\equiv 0$ and $v_1\equiv 1$ also holds for any $\gamma\in\left(0,\frac{n}{2}\right)$, since the corresponding extension problem shares many similarities with \eqref{eqfracesp} (c.f. \cite{MarChang,CaseChang,Ray}).
\end{remark}

The linearization at $V_1\equiv 1$ will be considered in Section \ref{section:linear}.

\subsection{A conserved Hamiltonian}

Here we give the proof of Theorem \ref{thctthamiltonian}. The idea comes from \cite{paperxavi}, where they consider layer solutions for semilinear equations with fractional Laplacian and a double-well potential. Multiply the first equation in \eqref{problemaradial} by $e(\rho)\partial_t V$, %
 and integrate with respect to $\rho\in(0,2)$, obtaining
\begin{equation*}
\begin{split}
-\int_{0}^{2} \partial_{\rho}\left(\rho^{a}e(\rho)\partial_{\rho}V\right)\partial_t V \,d\rho
-\int_{0}^{2}  \rho^{a}e_1(\rho)%
\partial_{tt}V\partial_{t}V\,d\rho+
\int_{0}^{2} \rho^a e_2(\rho)V\partial_t V\,d\rho=0,
\end{split}
\end{equation*}
where we have defined $e$, $e_1$, $e_2$ as in \eqref{e}.
We realize that $\partial_{tt}V\partial_{t}V=\frac{1}{2}\partial_t((\partial_tV)^2)$ and
$V\partial_t V=\frac{1}{2}\partial_t(V^2)$, thus integrating by parts in the first term above we get
\begin{equation*}
\begin{split}
&\int_{0}^{2}  \rho^{a}e(\rho)%
\partial_{\rho}V\partial_{t\rho} V
\,d\rho+\left(\rho^a e(\rho)\partial_{\rho}V\partial_tV\right)|_{\rho=0}\\&
-\partial_t\left(\tfrac{1}{2}\int_{0}^{2}  \rho^{a}e_1(\rho)(\partial_tV)^2\,d\rho\right)+
\partial_t\left(\tfrac{1}{2}\int_{0}^{2}\rho^a e_2(\rho)V^2\,d\rho\right)=0.
\end{split}
\end{equation*}
Here we have used the regularity of $V$ at $\rho=2$.
Again we note that $\partial_{\rho}V\partial_{t\rho} V=\frac{1}{2}\partial_t((\partial_{\rho}V)^2)$ and using the boundary condition, i.e., the third equation in \eqref{problemaradial},  %
we have
\begin{equation}\label{ham1}
\begin{split}
\tfrac{1}{2}\partial_t\left(\int_{0}^{2} \rho^{a}e(\rho)(\partial_{\rho}V)^2 \,d\rho\right)&
-\tfrac{1}{2}\partial_t\left(\int_{0}^{2}  \rho^{a}e_1(\rho)(\partial_tV)^2\,d\rho\right)\\
&+\tfrac{1}{2}\partial_t\left(\int_{0}^{2}  \rho^a e_2(\rho)V^2\,d\rho\right)=%
\frac{c_{n,\gamma}}{\tilde{d}\gamma}v^{\frac{n+2\gamma}{n-2\gamma}}\partial_tv .
\end{split}
\end{equation}
Define
 $$G(v)=C_{n,\gamma}v^{\frac{2n}{n-2\gamma}},$$
where the constant is defined in \eqref{constantC}.
In this way, we have from \eqref{ham1} that
\begin{equation*}
\begin{split}
&\tfrac{1}{2}\partial_t\int_{0}^{2} \left\{ \rho^{a}e(\rho)(\partial_{\rho}V)^2
-\rho^{a}e_1(\rho)(\partial_tV)^2
+ \rho^a e_2(\rho)V^2\right\}\,d\rho-\partial_t(G(v))=0.
\end{split}
\end{equation*}
So we can conclude that the Hamiltonian
\begin{equation*}
-H_\gamma(t):=\tfrac{1}{2}\int_{0}^{2}  \rho^{a}\left\{e(\rho)(\partial_{\rho}V)^2
-e_1(\rho)(\partial_tV)^2
+ e_2(\rho)V^2\right\}\,d\rho
-G(v),
\end{equation*}
is constant respect to $t$. This concludes the proof of Theorem \ref{thctthamiltonian}.

\begin{remark}
One can rewrite the Hamiltonian in terms of the defining function $\rho^*$. For this, we may follow similar computations as above but starting with equation \eqref{eqfracesp}. Indeed, let $V$ be a solution of \eqref{Yamfracspec}, then the new Hamiltonian quantity
\begin{equation}\label{Hamiltonian*}
\begin{split}
H_\gamma^*(t):=&\frac{c_{n,\gamma}}{\tilde{d}_{\gamma}}\left(\tfrac{n-2\gamma}{2n}v^{\frac{2n}{n-2\gamma}}
-\tfrac{1}{2}v^2\right)\\
&+\tfrac{1}{2}\int_0^{\rho^*_0}(\rho^*)^{a}
\left\{e_1^*(\rho)(\partial_tV)^2-e^*(\rho)(\partial_{\rho^*}V)^2\right\}\,d\rho^*
\end{split}\end{equation}
is constant respect to $t$. Here
$$e^*(\rho)=\left(\tfrac{\rho^*}{\rho}\right)^2,\quad e_1^*(\rho)=\left(\tfrac{\rho^*}{\rho}\right)^2 e_1(\rho).$$
This quantity $H_\gamma^*$ is the natural generalization of \eqref{Hamiltonian1}.
\end{remark}

Now we observe that in the local case, the Hamiltonian \eqref{Hamiltonian1} is a convex function in the domain we are interested, thus its level sets are well defined closed trajectories around the equilibrium $v_1\equiv 1$. We would like to have the analogous result for the Hamiltonian quantity $H^*_\gamma$ from \eqref{Hamiltonian*}. This is a very interesting open question that we conjecture to be true. In any case, the second variation for $H^*_\gamma$ near this equilibrium is:
\begin{equation*}\begin{split}
\left.\frac{d^2}{d\epsilon^2}\right|_{\epsilon=0}H^*_\gamma(V_1+\epsilon V)&=
\tfrac{c_{n,\gamma}}{\tilde d_\gamma}\tfrac{4\gamma}{n-2\gamma}v^2\\
&+\tfrac{1}{2}\int_0^{\rho_0^*} (\rho^*)^a\tfrac{(\rho^*)^2}{\rho^2} \left\{e_1(\rho)(\partial_t V)^2-e(\rho)(\partial_{\rho^*} V)^2\right\}d\rho^*.
\end{split}\end{equation*}

\section{The homoclinic solution}\label{section:explicit}

For this section we will take $\gamma\in\left(0,\tfrac{n}{2}\right)$, since it does not depend on the extension problem \eqref{equation1}. It is clear that the standard bubble \eqref{sphere} is a solution of equation \eqref{equation0} that has a removable singularity at the origin. Note that, because of our choice of the constant $c_{n,\gamma}$, we need to normalize it by a positive multiplicative constant. We prove here that, on the boundary phase portrait, the equilibrium $v_1\equiv 1$ stays always bounded by this homoclinic solution in a boundary phase portrait. More precisely:

\begin{proposition}
The positive function
\begin{equation}\label{C}
v_\infty(t)=C(\cosh t)^{-\frac{n-2\gamma}{2}},\quad \text{ with}\quad C=
\left(c_{n,\gamma}\frac{\Gamma(\frac{n}{2}-\gamma)}{\Gamma(\frac{n}{2}+\gamma)}\right)^{-\frac{n-2\gamma}
{4\gamma}}>1\equiv v_1,
\end{equation}
is a smooth solution of the fractional Yamabe problem \eqref{fracyamextv}. The value of $c_{n,\gamma}$ is given in Proposition \ref{cte}.
\end{proposition}

\begin{proof}
The canonical metric on the sphere, rescaled by a constant, maybe written as
\begin{equation*}\label{sphmetric}
g_C=C^{\frac{4}{n-2\gamma}}g_{\s^n}=[C(\cosh t)^{-\frac{n-2\gamma}{2}}]^{\frac{4}{n-2\gamma}}g_0.
\end{equation*}
We choose $C$ such that the fractional curvature of the standard sphere is normalized to
\begin{equation}\label{Qimpongo}
Q^{g_C}_{\gamma}\equiv c_{n,\gamma}.
\end{equation}
Now we use the conformal property \eqref{confQ} for the operator $P^{g_{\s^n}}_{\gamma}$:
\begin{equation}\label{PgC}
P_{\gamma}^{g_{\s^n}}(C)=C^{\frac{n+2\gamma}{n-2\gamma}}Q_{\gamma}^{g_C}.
\end{equation}
One checks that the fractional curvature is homogeneous of order $\gamma$ under rescaling of the metric. Indeed, because of \eqref{PgC} and the linearity of the operator $P_{\gamma}$
\begin{equation}\label{QC}
\begin{split}
Q_{\gamma}^{g_C}&=C^{-\frac{n+2\gamma}{n-2\gamma}}P_{\gamma}^{g_{\s^n}}(C)=C^{-\frac{(n+2\gamma)}
{n-2\gamma}+1}
P_{\gamma}^{g_{\s^n}}(1)=C^{-\frac{4\gamma}{n-2\gamma}}Q_{\gamma}^{g_{\s^n}}.\\
\end{split}
\end{equation}
Comparing equalities \eqref{Qimpongo} and \eqref{QC}, together with the value of the curvature on the standard sphere \eqref{fqsphere} we find the precise value of $C$ as claimed in \eqref{C}.

Next, let us check that the value of the constant $C$ is larger than one.
 Because of Proposition \ref{cte} we have to test that
$$2^{2\gamma}\left(\frac{\Gamma(\frac{1}{2}(\frac{n}{2}+\gamma))}{\Gamma(\frac{1}{2}(\frac{n}{2}-\gamma))}\right)^2
\frac{\Gamma(\frac{n}{2}-\gamma)}{\Gamma(\frac{n}{2}+\gamma)}<1.$$
Using the property \eqref{prop3g} of the Gamma function, given in Lemma \ref{propiedadesgamma}, we only need to verify that
\begin{equation*}
X(n,\gamma):=\frac{\Gamma(\frac{1}{2}(\frac{n}{2}+\gamma))}{\Gamma(\frac{1}{2}(\frac{n}{2}-\gamma))}\frac{\Gamma(\frac{1}{2}(\frac{n}{2}-\gamma)
+\frac{1}{2})}{\Gamma(\frac{1}{2}(\frac{n}{2}+\gamma)+\frac{1}{2})}<1.
\end{equation*}
Thanks to Lemma \ref{Xcrecimiento} below, it is enough to see that
\begin{equation*}
X(n,0)\leq 1\quad \forall n,
\end{equation*}
which holds trivially.
\end{proof}

\begin{lemma}\label{Xcrecimiento}
The function $X(n,\gamma)$ defined as follows
\begin{equation*}\label{Xng}
X(n,\gamma):=\frac{\Gamma(\frac{1}{2}(\frac{n}{2}+\gamma))}
{\Gamma(\frac{1}{2}(\frac{n}{2}-\gamma))}\frac{\Gamma(\frac{1}{2}(\frac{n}{2}-\gamma)+\frac{1}{2})}
{\Gamma(\frac{1}{2}(\tfrac{n}{2}+\gamma)+\tfrac{1}{2})},
\end{equation*}
is increasing in $n$, and decreasing in $\gamma$.
\end{lemma}

\begin{proof}
If we denote $\psi(z)$ the Digamma function from Lemma \ref{propiedadesgamma}, we can use the expansion \eqref{propdg} to study the growth of the function $X(n,\gamma)$ with respect to $n$ and $\gamma$.
First,
\begin{equation*}
\begin{split}\frac{\partial}{\partial n}(\log X(n,\gamma))&=\tfrac{1}{4}\left(\psi(\tfrac{n}{4}+\tfrac{\gamma}{2})+\psi(\tfrac{n}{4}-\tfrac{\gamma}{2}
      +\tfrac{1}{2})-\psi(\tfrac{n}{4}-\tfrac{\gamma}{2})-\psi(\tfrac{n}{4}+\tfrac{\gamma}{2}+\tfrac{1}{2})
      \right)  %
      \\&=\tfrac{\gamma}{4}\sum_{m=0}^\infty\tfrac{m+\frac{n}{4}+\frac{1}{4}}{\left[\left(m+\frac{n}{4}\right)^2
      -\frac{\gamma^2}{4}\right]\left[\left(m+\frac{n}{4}+\frac{1}{2}\right)^2
      -\frac{\gamma^2}{4}\right]}>0.
      \end{split}\end{equation*}
and
  \begin{equation}\label{formulon}
     \begin{split}\frac{\partial}{\partial \gamma}(\log X(n,\gamma))&=\tfrac{1}{2}\left(\psi(\tfrac{n}{4}+\tfrac{\gamma}{2})-\psi(\tfrac{n}{4}-\tfrac{\gamma}{2}
      +\tfrac{1}{2})+\psi(\tfrac{n}{4}-\tfrac{\gamma}{2})-\psi(\tfrac{n}{4}+\tfrac{\gamma}{2}+\tfrac{1}{2})
      \right)\\
      &=-\tfrac{1}{2}\sum_{m=0}^{\infty}\left[\tfrac{\left(m+\tfrac{n}{4}+\tfrac{1}{2}\right)
      \left(m+\tfrac{n}{4}\right)+\tfrac{\gamma^2}{4}}{\left((m+\tfrac{n}{4}+\tfrac{1}{2})^2-\tfrac{\gamma^2}
      {4}\right)\left((m+\tfrac{n}{4})^2-\tfrac{\gamma^2}{4}\right)}\right]<0.
\end{split}\end{equation}

\end{proof}

\section{Linear analysis}\label{section:linear}

Let us say a few words about stability. Let $v_*$ be a solution of \eqref{equation2}. The corresponding linearized equation  is
$$P_\gamma^{g_0} v=c_{n,\gamma}\tfrac{n+2\gamma}{n-2\gamma}\, v_*^{\frac{4\gamma}{n-2\gamma}}v.$$
We say that $v_*$ is a stable solution of \eqref{equation2} if
\begin{equation}\label{stability}
\int_M v (P^{g_0}_\gamma v)\dvol_{g_0}-c_{n,\gamma} \tfrac{n+2\gamma}{n-2\gamma}\int_{M}v_*^{\frac{4\gamma}{n-2\gamma}} v^2\dvol_{g_0}\geq 0,\quad \text{for all}\quad v\in\mathcal C_0^\infty(M).
\end{equation}
We observe here that the equilibrium $v_1\equiv 1$ is not a stable solution for \eqref{equation2} just by comparing the constant appearing in \eqref{stability} and in the Hardy inequality \eqref{Hardy-v}. In addition, one easily checks that the equilibrium solution $v_0\equiv 0$ is stable.

But it is more interesting to look at the explicit solution $v_\infty$ given in \eqref{C}. It follows from the Hardy inequality \eqref{Hardy-v} that this explicit solution is not stable. The kernel of the linearization at $v_\infty$ is calculated in \cite{Davila-delPino-Sire}, where they show that, although non-trivial, is non-degenerate, i.e., is generated by translations and dilations of the standard bubble.\\

Let us look more closely at the spectrum of the operator $P^{g_0}_{\gamma}$.
It is well known that $P^{g_0}_{\gamma}$ is self-adjoint (\cite{GrahamZorski}), and then %
we can compute its first eigenvalue through the Rayleigh quotient. Thus we minimize
\begin{equation*}
\inf_{v\in\mathcal C_0^\infty(M)}\frac{\int_MvP_{\gamma}^{g_0}v\dvol_{g_0}}{\int_{M}v^2\dvol_{g_0}},
\end{equation*}
where $M=\r\times\mathbb S^{n-1}$. We can apply Theorem 4.2 and Corollary 4.3 in \cite{MarQing} (or the Hardy inequality \eqref{Hardy-v}) to conclude that $P^{g_0}_{\gamma}$ is positive-definite. Moreover, the first eigenspace is of dimension one.\\

Now we consider the linear analysis around the equilibrium solution $v_1\equiv 1$. In order to motivate our results, let us explain what happens in the local case $\gamma=1$ for the linearization (see \cite{Mazzeo-Pacard:isolated,Mazzeo-Pollack-Uhlenbeck:moduli-spaces,KMPS}). In these papers the authors actually characterize the spectrum for the linearization of the equation
$$P_1^{g_0} v=\tfrac{(n-2)^2}{4} v^{\frac{n+2}{n-2}},$$
given by (after projection over each eigenspace  $\langle E_k\rangle$, $k=0,1,\ldots$)
$$-\ddot{v}-[n-2+\mu_k]v=0.$$
Note that this equation has periodic solutions only for $k=0$, of period $L_1=\frac{2\pi}{\sqrt{\lambda^0}}$ for $\lambda^0=n-2$. Thus we recover \eqref{L1}. For the rest of $k=1,\ldots,$ the corresponding $\lambda^k=n-2+\mu_k<0$, so we do not get periodic solutions.\\

The linearization of equation \eqref{equation2} around the equilibrium $v_1\equiv 1$ is given by
\begin{equation}\label{conforlinearizada}
\begin{split}
P^{g_0}_{\gamma}v=c_{n,\gamma}\tfrac{n+2\gamma}{n-2\gamma}v
\end{split}
\end{equation}
Here we will calculate the period of solutions for this linearized problem (for the projection $k=0$), as stated in Theorem \ref{theorem:linearization},  by the method of separation of variables. We also conjecture that there are not periodic solutions for the linearized problem \eqref{linearization-k} for the rest of $k=1,...$, as it happens in the classical clase.

Therefore, we consider the projection of  equation \eqref{eqs} over each eigenspace $\langle E_k\rangle$, $k=0,1,\ldots$. Let
$$U_k(z,t)=T(t)Z(z),$$
be a solution of \eqref{equk}. Then
\begin{equation*}
(1-z^2)\frac{Z''(z)}{Z(z)}+\left(\tfrac{n-1}{z}-z\right) \frac{Z'(z)}{Z(z)}+\frac{\frac{n^2}{4}-\gamma^2}{1-z^2}+\frac{\mu_k}{z^2}
=-\frac{T''(t)}{T(t)}=\lambda^k,
\end{equation*}
for a constant  $\lambda^k:=\lambda^k(\gamma)\in\mathbb R$. We are only interested in the case $\lambda>0$, which is the one that leads to periodic solutions in the variable $t$. The  period would be calculated from $L^k:=L^k(\gamma)=\frac{2\pi}{\sqrt {\lambda^k}}$.

 Note that the equation for $Z(z)$ is simply \eqref{equkfou} with $\xi^2$ replaced by $\lambda^k$. From the discussion in Section \ref{simbolo}, in particular \eqref{varphiz}, \eqref{A} and \eqref{S(s)} we have that
\begin{equation*}
\begin{split}
Z(z)=&(1+z)^{\frac{n}{4}-\frac{\gamma}{2}}(1-z)^{\frac{n}{4}-\frac{\gamma}{2}}z^{1-\frac{n}{2}
+\sqrt{(\tfrac{n}{2}-1)^2-\mu_k}} \\ &
\Hyperg(a,b;a+b-c+1;1-z^2)\\
+&\kappa(1+z)^{\frac{n}{4}+\frac{\gamma}{2}}(1-z)^{\frac{n}{4}+\frac{\gamma}{2}}z^{1-\frac{n}{2}
-\sqrt{(\tfrac{n}{2}-1)^2-\mu_k}}\\&
\Hyperg(c-a,c-b;c-a-b+1;1-z^2),
\end{split}
\end{equation*}
where \begin{align*}
a&=-\tfrac{\gamma}{2}+\tfrac{1}{2}+\tfrac{\sqrt{(\tfrac{n}{2}-1)^2-\mu_k}}{2}
            +i\tfrac{\sqrt{\lambda^k}}{2}\\
b&=-\tfrac{\gamma}{2}+\tfrac{1}{2}+\tfrac{\sqrt{(\tfrac{n}{2}-1)^2-\mu_k}}{2}
            -i\tfrac{\sqrt{\lambda^k}}{2},\\
c&=1+\sqrt{(\tfrac{n}{2}-1)^2-\mu_k},\\ \kappa&=\tfrac{\Gamma(-\gamma)\left|\Gamma\large(\tfrac{1}{2}+\tfrac{\gamma}{2}
+\tfrac{\sqrt{(\tfrac{n}{2}-1)^2-\mu_k}}{2}
    +i\frac{\sqrt{\lambda^k}}{2}\large)\right|^2}{\Gamma(\gamma)\left|\Gamma\large(\tfrac{1}{2}
    -\tfrac{\gamma}{2}
    +\tfrac{\sqrt{(\tfrac{n}{2}-1)^2-\mu_k}}{2}+i\tfrac{\sqrt{\lambda^k}}{2}\large)\right|^2}.
\end{align*}
We use the change of variable \eqref{zrho} to analyze the asymptotic behavior of $Z$ near the conformal infinity $\rho=0$
\begin{equation*}
Z\sim\rho^{\frac{n}{2}-\gamma}+\kappa\rho^{\frac{n}{2}+\gamma}.
\end{equation*}
From the definition of the scattering operator \eqref{formau}, \eqref{scattering}, and the definition of the conformal fractional Laplacian we have that
\begin{equation*}
P^{k}_{\gamma}v_k=2^{2\gamma}\frac{\left|\Gamma(\tfrac{1}{2}+\tfrac{\gamma}{2}
+\tfrac{\sqrt{(\tfrac{n}{2}-1)^2-\mu_k}}{2}+\tfrac{\sqrt{\lambda^k}}{2}i)\right|^2}
{\left|\Gamma(\tfrac{1}{2}-\tfrac{\gamma}{2}
+\tfrac{\sqrt{(\tfrac{n}{2}-1)^2-\mu_k}}{2}+\tfrac{\sqrt{\lambda^k}}{2}i)\right|^2}
v.
\end{equation*}
Imposing the boundary condition
\eqref{conforlinearizada} and the value of $c_{n,\gamma}$ given in \eqref{cng}, the unknown $\lambda^k$ must be a solution of
\begin{equation}\label{eqlambdak}
\frac{\left|\Gamma(\tfrac{1}{2}+\tfrac{\gamma}{2}+\tfrac{\sqrt{(\tfrac{n}{2}-1)^2-\mu_k}}{2}
+\tfrac{\sqrt{\lambda^k}}{2}i)\right|^2}
{\left|\Gamma(\tfrac{1}{2}-\tfrac{\gamma}{2}+\tfrac{\sqrt{(\tfrac{n}{2}-1)^2-\mu_k}}{2}
+\tfrac{\sqrt{\lambda^k}}{2}i)\right|^2}
=\frac{n+2\gamma}{n-2\gamma}\frac{\left|\Gamma\left(\frac{1}{2}\left(\frac{n}{2}
+\gamma\right)\right)\right|^2}
{\left|\Gamma\left(\frac{1}{2}\left(\frac{n}{2}-\gamma\right)\right)\right|^2}.
\end{equation}
Note that for the canonical projection $k=0$, equality \eqref{eqlambdak} simplifies to
\begin{equation}\label{eqlambda2}
\frac{\left|\Gamma(\tfrac{n}{4}+\tfrac{\gamma}{2}+\tfrac{\sqrt{\lambda^0}}{2}i)\right|^2}
{\left|\Gamma(\tfrac{n}{4}-\tfrac{\gamma}{2}+\tfrac{\sqrt{\lambda^0}}{2}i)\right|^2}
=\frac{n+2\gamma}{n-2\gamma}\frac{\left|\Gamma\left(\frac{1}{2}\left(\frac{n}{2}+\gamma\right)\right)\right|^2}
{\left|\Gamma\left(\frac{1}{2}\left(\frac{n}{2}-\gamma\right)\right)\right|^2}.
\end{equation}
This equation \eqref{eqlambda2} lets us recover the value of $\lambda^0$ for the classical case $\gamma=1$. Indeed, using property \eqref{prop2g} we get $\lambda^0=n-2$ and we recover \eqref{L1}, where $L_1:=L^0(1).$

Going back to equation \eqref{eqlambdak} we can assert that the value of $\lambda^k$ can not be zero and it is unique for each $k$. Indeed if $\lambda=0$ we get a contradiction, and if $\lambda>0$ we may proceed as follows. Define
 $$F(\beta)=\frac{\frac{|\Gamma(\alpha_k+\beta i)|^2}{|\Gamma(\tilde{\alpha}_k+\beta i)|^2}}{\frac{n+2\gamma}{n-2\gamma}\frac{\left|\Gamma\left(\frac{1}{2}\left(\frac{n}{2}+\gamma\right)\right)
 \right|^2}{\left|\Gamma
 \left(\frac{1}{2}\left(\frac{n}{2}-\gamma\right)\right)\right|^2}},$$
 where
$$\alpha_k=\tfrac{1}{2}+\tfrac{\gamma}{2}+\tfrac{\sqrt{(\tfrac{n}{2}-1)^2-\mu_k}}{2},\quad
\tilde{\alpha}_k=\tfrac{1}{2}-\tfrac{\gamma}{2}+\tfrac{\sqrt{(\tfrac{n}{2}-1)^2-\mu_k}}{2}
\quad\text{and}\quad \beta=\tfrac{\sqrt{\lambda^k}}{2}.$$
Note that equation \eqref{eqlambdak} is written as $F(\beta)=1$, for some $\beta>0$.
We derive this expression with respect to $\beta$,
\begin{equation*}
(\log F(\beta))'=\frac{2}{\frac{n+2\gamma}{n-2\gamma}\frac{\left|\Gamma\left(\frac{1}{2}\left(\frac{n}{2}
+\gamma\right)\right)
 \right|^2}{\left|\Gamma
 \left(\frac{1}{2}\left(\frac{n}{2}-\gamma\right)\right)\right|^2}}
 \,\Im [\psi(\tilde{\alpha}_k+\beta i)-\psi(\alpha_k+\beta i)].
\end{equation*}
 Here $\Im$ represents the imaginary part of a complex number and $\psi(z)$ the Digamma function from Lemma \ref{propiedadesgamma}. We can use the expansion \eqref{propdg} to arrive at
\begin{equation*}
(\log F(\beta))'
=c\sum_{m=0}^{\infty}
\tfrac{\gamma\beta\left(2m+1+\sqrt{(\tfrac{n}{2}-1)^2-\mu_k}\right)}
{\left[\left(m+\tfrac{1}{2}+\tfrac{\sqrt{(\tfrac{n}{2}-1)^2-\mu_k}}{2}\right)^2
-{\beta}^2-\tfrac{\gamma^2}{4}\right]^2
+\left[\left(2m+1+\sqrt{(\tfrac{n}{2}-1)^2-\mu_k}\right)\beta\right]^2},
\end{equation*}
for some positive constant $c$. Therefore $F(\beta)$ is an increasing function of $\beta$.

Next, note that
$$\lim_{\beta\to +\infty} F(\beta)=+\infty,$$
for all $k=0,1,\ldots$. This follows easily writing
$$\frac{n+2\gamma}{n-2\gamma} F(\beta)=\frac{B\left(\alpha_k+\beta i, \frac{n}{4}-\frac{\gamma}{2}\right)}{B\left(\tilde \alpha_k+\beta i, \frac{n}{4}+\frac{\gamma}{2}\right)},$$
and the asymptotic behavior for the Beta function \eqref{propbeta} from Lemma \ref{propiedadesgamma}.

Now we look at the projection $k=0$. One immediately calculates
$$F(0)=\frac{n-2\gamma}{n+2\gamma}<1,$$
so there exists (and it is unique) a solution $\lambda^0=\lambda^0(\gamma)>0$ for the equation $F(\beta)=1$. From the proof one also gets that
$$\lim_{\gamma\to 1} \lambda^0(\gamma)=n-2.$$
This concludes the proof of Theorem \ref{theorem:linearization}.\\

We believe that, as in the classical case $F(\beta)=1$ does not have any positive solution for $k=1,2,\ldots$. This is a well supported conjecture that only depends on making more rigorous some numerical analysis. In order to motivate this conjecture, let us try to show that
 $f_k>1$ for $k=1,2,\ldots$, where we have defined
$$F(0)=\frac{(n-2\gamma)|\Gamma(\alpha_k)|^2\left|\Gamma
 \left(\frac{1}{2}\left(\frac{n}{2}-\gamma\right)\right)\right|^2}
 {(n+2\gamma)|\Gamma(\tilde{\alpha}_k)|^2
\left|\Gamma\left(\frac{1}{2}\left(\frac{n}{2}+\gamma\right)\right)
 \right|^2}=:f_k.$$
Using the same ideas as above, one checks that $f_k$ is an increasing function of $k$, and it is enough to show that
 $$f_1= \frac{(n-2\gamma)\left|\Gamma\left(\frac{1}{2}+\frac{\gamma}{2}+\frac{n}{4}\right) \Gamma\left(\frac{n}{4}-\frac{\gamma}{2}\right) \right|^2}{(n+2\gamma)\left|\Gamma\left(\frac{1}{2}-\frac{\gamma}{2}+\frac{n}{4}\right) \Gamma\left(\frac{n}{4}+\frac{\gamma}{2}\right) \right|^2}=\frac{n-2\gamma}{n+2\gamma}X(n,\gamma)^{-2}>1,$$
where $X(n,\gamma)$ is defined in Lemma \ref{Xcrecimiento}. We have numerically observed that $f_1=f_1(\gamma)$ is an increasing function in $\gamma$. Since for $\gamma=0$ we already have that  $f_1(0)=1$, we would conclude that $f_k> f_1\geq 1$, as desired.
\section*{Acknowledgement}
The authors would like to thank the hospitality of Princeton University, where part of this work was carried out. They also thank Xavier Cabr{\'e}, Marco Fontelos, Robin Graham, Xavier Ros-Oton and Joan Sol{\`a}-Morales for their useful conversations and suggestions.
\bibliographystyle{abbrv}

\end{document}